\documentclass[a4paper,oneside, reqno]{amsart}
\usepackage{mathrsfs}
\usepackage{amsfonts}
\usepackage{amssymb}
\usepackage{amsxtra}
\usepackage{dsfont}
\usepackage{amsthm}
\usepackage{ulem} 

\usepackage{orcidlink}

\usepackage[english,polish]{babel}

\newcommand{\pl}[1]{\foreignlanguage{polish}{#1}}

\usepackage{xcolor}

\usepackage{enumitem}
\usepackage{hyperref}
\usepackage{stmaryrd} 

\theoremstyle{plain}
\newtheorem{theorem}{Theorem}[section]
\newtheorem*{sq}{Stein's question}
\newtheorem{proposition}{Proposition}[section]
\newtheorem{corollary}[proposition]{Corollary}
\newtheorem{lemma}[proposition]{Lemma}
\newtheorem{condition}[proposition]{Condition}
\newtheorem*{dfc}{Dimension-free conjecture}
\theoremstyle{definition}

\newtheorem{remark}{Remark}[section]
\newtheorem*{remark*}{Remark}

\numberwithin{equation}{section}

\newcommand{\RR}{\mathbb{R}}
\newcommand{\ZZ}{\mathbb{Z}}
\newcommand{\TT}{\mathbb{T}}
\newcommand{\CC}{\mathbb{C}}
\newcommand{\NN}{\mathbb{N}}

\newcommand{\GG}{U}
\newcommand{\II}{\mathbb{I}}

\newcommand{\calB}{\mathcal{B}}

\newcommand{\ind}[1]{{\mathds{1}_{{#1}}}}

\newcommand{\abs}[1]{{\lvert {#1} \rvert}}

\newcommand{\8}{\infty}

\subjclass[2020]{42B25, 42B15}
\keywords{discrete maximal function, dimension-free estimates}

\title[Dimension-free estimates for discrete maximal functions of sampled gaussians]
{Dimension-free estimates for discrete maximal functions related to normalized sampled gaussians}

\author{Mariusz Mirek \orcidlink{0000-0001-7641-9893}}
\address{Mariusz Mirek \\
  Department of Mathematics\\
  Rutgers University\\
Piscataway, NJ 08854\\ USA \&
	Instytut Matematyczny\\
	Uniwersytet \pl{Wroc{\l}awski}\\
	Plac Grun\-waldzki 2\\
	50-384 \pl{Wroc{\l}aw}\\
	Poland}
\email{mariusz.mirek@rutgers.edu}

\author{Tomasz Z. Szarek \orcidlink{0000-0003-0821-5607} }
\address{ Tomasz Z. Szarek\\
Department of Mathematics \\
University of Georgia \\
Athens, GA 30602 \\ 
USA  \&
Instytut Matematyczny\\
Uniwersytet \pl{Wroc{\l}awski}\\
Plac Grun\-waldzki 2\\
50-384 \pl{Wroc{\l}aw}\\
Poland}
\email{tzszarek@uga.edu}

\author[Wr{\'o}bel]{B{\l}a{\.z}ej Wr{\'o}bel \orcidlink{0000-0003-0413-8931}}
\address[B{\l}a{\.z}ej Wr{\'o}bel]{
	Instytut Matematyczny Polskiej Akademii Nauk\\
	Śniadeckich 8\\
	00-656 Warszawa\\
	Poland \& Instytut Matematyczny\\
	Uniwersytet \pl{Wroc{\l}awski}\\
	Plac Grun\-waldzki 2\\
	50-384 \pl{Wroc{\l}aw}\\
	Poland}
\email{blazej.wrobel@math.uni.wroc.pl}

\thanks{Mariusz Mirek was partially supported by the NSF grant
(DMS-2154712) and by the NSF CAREER grant (DMS-2236493).  Tomasz
Z. Szarek and B{\l}a{\.z}ej Wr{\'o}bel were supported by the National
Science Centre, Poland, grant Sonata Bis 2022/46/E/ST1/00036. For the purpose of Open Access, the authors have applied a
CC-BY public copyright licence to any Author Accepted Manuscript (AAM) version arising from this submission.}

\begin{document}

 
\selectlanguage{english}

\begin{abstract}
In this paper, we investigate dimension-free estimates for convolution operators with discrete normalized sampled Gaussians
(related to the Theta function) in the context of maximal, jump,
$r$-variational and oscillation inequalities on $\ell^p(\mathbb{Z}^d)$
spaces. This is the first instance of a discrete maximal function in
the literature --- a non-semigroup example --- where dimension-free
$\ell^p(\mathbb{Z}^d)$ bounds are provided for the entire range of
$1 < p < \infty$. The methods of proof rely on developing robust
Fourier techniques, which are combined with the fractional derivative, a
tool that has not been previously applied to studying similar
questions in the discrete setting.
\end{abstract}

\maketitle

\section{Introduction}
\label{sec:1}

\subsection{Statement of the main results}
Let  $\TT:=\RR/\ZZ$ denote one-dimensional torus, which will be identified with  $[-1/2,1/2)$. For $t>0$ and $\zeta \in \TT$ we define a one-dimensional Theta function by
\begin{align*}
\theta_t(\zeta):=\sum_{k\in \ZZ}\exp(-\pi(k-\zeta)^2 t).
\end{align*}
For $\xi \in \TT^d$ we also define its  $d$-dimensional variant by
\begin{equation}
\label{eq:Theta_def}
\Theta_t(\xi):=\prod_{j\in[d]}\theta_t(\xi_j)=\sum_{k\in\ZZ^d} \exp(-\pi|k-\xi|^2t),
\end{equation}
where $|\cdot|$ denotes the Euclidean norm.

For $t>0$  the discrete normalized Gaussian kernel $g_t\colon \ZZ^d \to \RR$ is defined by
\begin{align*}
g_t(n) := \frac{1}{\Theta_{1/t}(0)}\exp\bigg(- \frac{\pi|n|^2}{t}\bigg), \qquad n\in \ZZ^d,
\end{align*}
and the corresponding convolution operator on $\ZZ^d$ with kernel $g_t$ is defined by 
\begin{equation*}
G_t f(x):= g_t*f(x)=\sum_{n\in \ZZ^d}g_t(n)f(x-n),\qquad x\in \ZZ^d, \quad f\in \ell^1(\ZZ^d).
\end{equation*}
We immediately see that $G_t$ is a
contraction on $\ell^p(\ZZ^d)$ spaces for all $1\le p\leq \infty,$ that is
\begin{equation}
\label{eq:Gtcontr}
\sup_{t>0}\|G_t \|_{\ell^p\to \ell^p}\le 1,
\end{equation}
which follows from Minkowski's convolution inequality and the fact that $\|g_t\|_{\ell^1}=1$ for any $t>0$.
The main goal of this paper is to prove the following dimension-free maximal, jump, variational and oscillation estimates. See Section~\ref{sec:var} for the definitions of these objects.
\begin{theorem} \label{thm:main}
For each $p\in(1,\infty)$ there exists $C_{p}>0$ independent of the dimension $d\in\ZZ_{+}$ such that
\begin{align} \label{jumpest}
\sup_{\lambda>0} \big\| \lambda N_{\lambda}(G_t f : t > 0)^{1/2} \big\|_{\ell^p} 
\le 
C_{p} \|f\|_{\ell^p}, \qquad f\in \ell^p(\ZZ^d).
\end{align}	
In particular, 	
for each $p\in(1,\infty)$ and $r \in (2,\infty)$ there exists $C_{p,r}>0$ independent of the dimension $d\in\ZZ_+$  such that
\begin{align*}
\big\| V^r(G_t f : t>0) \big\|_{\ell^p} \le C_{p,r} \|f\|_{\ell^p}, \qquad f\in \ell^p(\ZZ^d).
\end{align*}
Further, for each $p\in(1,\infty)$ there exists $C_p>0$ independent of the dimension $d\in\ZZ_+$  such that
\begin{align} \label{oscest}
\sup_{J\in\ZZ_+}\sup_{I\in \mathfrak{S}_J((0,\infty))} 	
\big\| O_{I, J}^{2}( G_t f : t > 0) \big\|_{\ell^p} \le C_{p} \|f\|_{\ell^p}, \qquad f\in \ell^p(\ZZ^d).
\end{align}
Furthermore, for each $p\in(1,\infty)$ there exists $C_p>0$ independent of the dimension $d\in\ZZ_+$  such that
\begin{align} \label{id:A}
\big\|\sup_{t > 0}|G_t f| \big\|_{\ell^p}\le C_p \|f\|_{\ell^p}, \qquad f\in \ell^p(\ZZ^d).
\end{align}
\end{theorem}

\subsection{Historical background and our motivations}
In order to explain why we are undertaking Theorem \ref{thm:main}, we
need to briefly describe the historical background of dimension-free
estimates in harmonic analysis. Let $\GG$ be a closed, bounded,
symmetric, and convex subset of $\RR^d$ with non-empty interior, which
we simply call a symmetric convex body. We set
$\GG_t:=\{y\in\RR^d: t^{-1}y\in \GG\}$ for every $t>0$ and, for every
$x\in\RR^d$ and $f\in L^1_{\rm loc}(\RR^d)$, we define the
Hardy--Littlewood averaging operator by
\begin{align}
\label{eq:112}
M_t^{\GG}f(x):=\frac{1}{|\GG_t|}\int_{\GG_t}f(x-y) \, dy.
\end{align}
A typical choice is $\GG=B^q(d)$ for $q\in[1, \infty]$, where $B^q(d)$
is the unit ball induced by the $\ell^q$ norm in $\RR^d$, see \eqref{eq:1} and \eqref{eq:1a}. Specifically,
$B^2(d)$ is the Euclidean unit ball and $B^\infty(d)$ is the cube in
$\RR^d$.  For $\emptyset\neq\mathbb I\subseteq \RR_+$ and
$p\in(1, \infty]$, let $0<C(p,\mathbb I, \GG)\le \infty$ be the smallest
constant for the following maximal inequality
\begin{align*}
\big\|\sup_{t\in\mathbb I}|M_t^{\GG}f|\big\|_{L^p(\RR^d)}\le C(p,\mathbb I, \GG)\|f\|_{L^p(\RR^d)},
\qquad f\in L^p(\RR^d).
\end{align*}
Using a standard covering argument (Vitali covering lemma) for $p=1$ and interpolation with $p=\infty$, it is not hard
to see that $C(p,\RR_+, \GG)<\infty$ for every $p\in(1, \infty]$ since $C(\infty,\RR_+, \GG)=1$.

Similarly, for every $t > 0$ let
$\GG_t:=\{y\in \RR^d\colon t^{-1}y\in \GG\}$ denote the
dilate of a set $\GG\subseteq \RR^d$.  For $x\in\ZZ^d$, we
define a discrete analogue of \eqref{eq:112} by
\begin{align*}
	\mathcal M_t^{\GG}f(x):=\frac{1}{|\GG_t\cap \ZZ^d|}\sum_{y\in \GG_t\cap\ZZ^d}f(x-y),
	\qquad
	f\in\ell^1(\ZZ^d),
\end{align*}
where the right-hand side above is understood as $0$ if $\GG_t\cap \ZZ^d = \emptyset$.
The corresponding maximal operator is given by $\mathcal M^{\GG} f(x) := \sup_{t>0} |\mathcal M_t^{\GG}f(x)|$.
For $\emptyset\neq\mathbb I\subseteq \RR_+$, $\GG\subseteq \RR^d$ and $p\in(1, \infty]$, let
$0<\mathcal C(p,\mathbb I, \GG)\le \infty$ be the smallest constant for
the following maximal inequality
\begin{align*}
\big\|\sup_{t\in\mathbb I}|\mathcal M^{\GG}_tf|\big\|_{\ell^p(\ZZ^d)}
\le \mathcal C(p,\mathbb I, \GG)\|f\|_{\ell^p(\ZZ^d)},
\qquad
f\in\ell^p(\ZZ^d).
\end{align*}
If $U \subseteq \RR^d$ is a symmetric convex body, then 
similar arguments as in the continuous setting imply that  $\mathcal C(p,\RR_+, \GG)<\infty$ for every $p\in(1, \infty]$.

\medskip

A major open problem in this area since the late 1970's arose from
Stein's work \cite{SteinMax} and asserts that:
\begin{dfc}
For every $p\in(1, \infty]$
there is a constant $C_p>0$ such that
\begin{align}
\label{eq:36}
\sup_{d\in\ZZ_+}\sup_{\GG\in\mathfrak B(d)}C(p,\RR_+, \GG)\le C_p
\end{align}
where $\mathfrak B(d)$ is the set of all convex symmetric bodies in $\RR^d$.
\end{dfc}

This problem has been studied for four decades by several authors.  We
now briefly list the current state of the art concerning
\eqref{eq:36}.

\smallskip

\begin{enumerate}[label*={\arabic*}.]
\item Stein \cite{SteinMax} (see also Stein and Str\"omberg
\cite{StStr}) proved that there exists a constant $C_p>0$ depending
only on $p\in(1, \infty]$ such that
$\sup_{d\in \ZZ_{+}}C(p,\RR_+, B^2(d))\le C_p$. The key idea from
\cite{SteinMax, StStr} is based on the method of rotations, which
enables one to view high-dimensional spheres as an average of rotated
low-dimensional ones. 
This, in turn, allows estimates for the spherical maximal
function to be transferred to higher dimensions without increasing the
constants.
Unfortunately, this method is limited to the Euclidean balls. 

\smallskip

\item Bourgain \cite{B4,B5} and Carbery
\cite{Car1} independently proved \eqref{eq:36} for $p\in(3/2, \infty]$. They also
showed a dyadic variant of \eqref{eq:36} for $p\in(1, \infty]$ with
$\mathbb D=\{2^n:n\in\ZZ\}$ in place of $\RR_+$. Although inequality \eqref{eq:36}  for $p\in(1, 3/2]$
still remains open,  the case of $q$-balls is quite well understood.

\smallskip

\item For the $q$-balls $B^q(d)$, we know that for every
$p\in(1, \infty]$ and $q\in[1, \infty]$, there is a constant
$C_{p, q}>0$ such that
$\sup_{d\in \ZZ_{+}}C(p,\RR_+, B^q(d))\le C_{p, q}$. This was established
by M\"uller in \cite{Mul1} (for $q\in [1, \infty)$) and by Bourgain in
\cite{B6} (for cubes $q=\infty$).

\smallskip

\item The first and third authors with Bourgain and Stein \cite{BMSW1}
extended the results from \cite{B4, B5, B6, Car1, Mul1} to
$r$-variational estimates with bounds independent of the dimension.
We also found new proofs of maximal theorems from \cite{B4, B5,
Car1}. The methods from \cite{BMSW1} also shed new light on
dimension-free estimates in the discrete setup, which is a new and
interesting direction.

\end{enumerate}

Several years ago the first and third authors in a collaboration with Bourgain and
Stein \cite{BMSW2, BMSW3, BMSW4} initiated the study of dimension-free estimates in the discrete setting. This line of research was partially motivated by a question of E.M.\ Stein \cite{SteinPR}  from the
mid 1990's about the dimension-free estimates for the discrete
Hardy--Littlewood maximal functions corresponding to the Euclidean
balls $B^2(d)$.
\begin{sq}
Is it 
	true that there is a universal constant $C>0$ such that
	\begin{align}
		\label{eq:29}
		\sup_{d\in\ZZ_{+}}\mathcal C(2,\RR_+, B^2(d))=\sup_{d\in\ZZ_{+}}\mathcal C(2,\sqrt{\NN}, B^2(d))\le C\, ?
	\end{align}
\end{sq}
We now briefly mention existing results in the discrete setting and their relation to \eqref{eq:36} and \eqref{eq:29}.
\begin{enumerate}[label*={\arabic*}.]

\item In \cite{BMSW3}, it was shown
that \eqref{eq:36} cannot be expected in the discrete setup.
Surprisingly, in contrast to the continuous setup \cite{B4, B5, B6, BMSW1,
Car1, Mul1, SteinMax, SteinRiesz}, one can prove
\cite{BMSW3} that for every $p\in(1, \infty)$, there is a constant
$C_p>0$ such that for certain ellipsoids $E(d)\subset\RR^d$ and all
$d\in \ZZ_+$, one has
\begin{align}
\label{eq:13}
\mathcal C(p,\RR_+, E(d)) \ge C_p(\log d)^{1/p}.
\end{align}
On the other hand, for the cubes $B^\infty(d)$, it was also shown in
\cite{BMSW3} that for every $p\in(3/2, \infty]$, there is a constant $C_p>0$
such that
$\sup_{d\in \ZZ_{+}}\mathcal C(p,\RR_+, B^{\infty}(d))\le C_p$. For
$p\in(1, 3/2]$, it still remains open whether
$\sup_{d\in \ZZ_{+}}\mathcal C(p,\RR_+, B^{\infty}(d))$ is finite. However,
one can prove that for all $p\in(1, \infty]$, there is a constant
$C_p>0$ such that
$\sup_{d\in \ZZ_{+}}\mathcal C(p,\mathbb D, B^{\infty}(d))\le C_p$.

\smallskip
\item In \cite{BMSW2, BMSW4}, similar questions in the context of
discrete Euclidean balls were studied, and it was shown that there is a universal constant $C>0$ such that for every
$p\in[2, \infty]$ we have
\begin{align}
\label{eq:10}
\sup_{d\in\ZZ_+}\mathcal C(p,\mathbb D, B^{2}(d))\le C.
\end{align}
A similar result was also obtained by the authors \cite{MiSzWr} in the
context of discrete Euclidean spheres. However, the general problem for $\RR_+$
in place of $\mathbb D$ remains open and thus Stein's question \eqref{eq:29} is still not answered.

\smallskip

\item Recently, the first and third authors in a collaboration with
Kosz and Plewa \cite{KMPW} proved that for every
$\GG\in\mathfrak B(d)$, one has
\begin{align}
\label{eq:51}
C(p,\RR_+, \GG)\le \mathcal C(p,\RR_+, \GG)
\quad \text{ for all } \quad
p\in[1, \infty],
\end{align}
where the case $p=1$ corresponds to the optimal constants in the weak
type $(1, 1)$ inequalities respectively in $\RR^d$ and $\ZZ^d$. It was
also shown in \cite{KMPW} that
$C(1,\RR_+, B^{\infty}(d))= \mathcal C(1,\RR_+, B^{\infty}(d))$,
which, in view of Aldaz's result \cite{Ald1}, yields that
$\mathcal C(1,\RR_+, B^{\infty}(d))\ _{\overrightarrow{d\to\infty}}\ \infty$. Thus
the boundedness of $\sup_{d\in \ZZ_{+}}\mathcal C(p,\RR_+, B^{\infty}(d))$
for $p\in(1, 3/2]$ cannot be deduced by interpolation from
\cite{BMSW3} for $p\in(3/2, \infty]$.  Inequality \eqref{eq:51}
exhibits a well known phenomenon in harmonic analysis that states that
it is harder to establish bounds for discrete operators than the
bounds for their continuous counterparts, and this is the best that we
could prove in this generality.
\end{enumerate}

Our motivations for investigating Theorem~\ref{thm:main}, on the one
hand, come from the articles \cite{BMSW2} and \cite{MiSzWr}, where
dimension-free estimates for the dyadic versions of the maximal
operators corresponding to the discrete Euclidean balls and spheres
were proved for all $p \in [2,\infty]$, see \eqref{eq:10}. On the
other hand, our motivations come from \cite{BMSW3}, where certain
ellipsoids satisfying \eqref{eq:13} have been constructed to
demonstrate that the dimension-free phenomena in the discrete setting
is not as broad as in the continuous setting. 

Taking into account the results from \cite{BMSW2, BMSW3} and
\cite{MiSzWr} and the observation that Euclidean balls are not
significantly different from ellipsoids, it was tempting to prove
\eqref{eq:13} with Euclidean balls $B^2(d)$ in place of the ellipsoids
$E(d)$.  At some point, we believed that the maximal function
appearing in  \eqref{id:A} could provide a negative answer to Stein's question \eqref{eq:29}. This was caused by a  simple pointwise
bound
\begin{align}
\label{eq:14}
\sup_{t > 0}|G_t f (x)| \le \mathcal{M}^{B^2(d)} f(x) 
\le \max\{ \mathcal{M}^{S^2(d)} f(x), f(x)\}, \qquad x\in \ZZ^d, \quad f\ge0,
\end{align}
where $S^2(d) := \{x \in \mathbb{R}^d : |x|_2 = 1\}$ is the Euclidean
unit sphere.  Inequality \eqref{eq:14} immediately guarantees that
dimension-free $\ell^p(\ZZ^d)$ estimates for $\mathcal{M}^{S^2(d)}f$
imply dimension-free $\ell^p(\ZZ^d)$ estimates for
$\mathcal{M}^{B^2(d)}f$ and the latter imply dimension-free
$\ell^p(\ZZ^d)$ estimates for $\sup_{t > 0}|G_t f|$. However,
inequality \eqref{eq:14} can also serve to disprove dimension-free
estimates for the $\mathcal{M}^{B^2(d)}f$ if it can be demonstrated
that there are no dimension-free estimates for $\sup_{t > 0}|G_t f|$
on $\ell^p(\ZZ^d)$ spaces. That was our hope, but this paper shows
that we were mistaken, as dimension-free estimates for
$\sup_{t > 0}|G_t f|$ are available for all $\ell^p(\ZZ^d)$ spaces
whenever $1<p<\infty$, see Theorem \ref{thm:main}.  Now our main
result, Theorem~\ref{thm:main}, may provide evidence that it could be
possible to obtain dimension-free estimates for
$\mathcal{M}^{B^2(d)}f$ and $\mathcal{M}^{S^2(d)}f$. Further evidence
supporting this supposition has been provided very recently by the
third author in collaboration with Niksiński \cite{NW}. Namely, it was
proved there that for every $\varepsilon>0$ there is a constant
$C_{\varepsilon}>0$ depending only on $\varepsilon$ and such that for
every $p\in[2, \infty]$ one has
\begin{align}
\label{eq:NW}
\sup_{d\in\ZZ_+}\mathcal C(p,(0,d^{1/2-\varepsilon}], B^{2}(d))\le C_{\varepsilon}.
\end{align}
Both our main result, Theorem \ref{thm:main}, and inequality
\eqref{eq:NW} strongly suggest that Stein's question \eqref{eq:29} may
indeed have a positive answer. We hope to explore these topics in
future work.

\subsection{Ideas of the proof and its corollaries}
The proof of Theorem~\ref{thm:main} is based on two main
ingredients. The first of these ingredients is the Stein complex
interpolation theorem \cite{bigs}. In contrast to other discrete
dimension-free problems \cite{BMSW2, BMSW3, BMSW4, MiSzWr}, here the
fractional derivative is a very useful tool that allows us to
implement Stein' complex interpolation theorem.  The second of these
ingredients, and perhaps the most important one from the perspective
of our article, are estimates for the Fourier transform $\mathcal F_{\ZZ^d}{g}_{t}$ of the kernel
$g_t$. As in several previous works \cite{B4,B5,B6,BMSW2,MiSzWr,Mul1}, such estimates 
serve as the key input that allows one to compare the operators under
consideration with another family of operators for which dimension-free bounds
are already known.
In those works, the authors compared their operators to heat semigroups, Poisson
semigroups, or martingale transforms, for which dimension-free maximal, jump,
variational, and oscillation estimates are available.
In our case, we compare $G_t$ with the heat semigroup associated with the discrete Laplacian,
for which dimension-free maximal, jump, variational, and oscillation estimates
have been established, see \cite{BMSW3, JR, MSZ1}.
It is worth mentioning that $G_t$ as well as its natural local
modification $G_{\psi^{-1}(t)}$ for $t \in (0,1)$, see
\eqref{def:psi}, are not (local) semigroups. This makes the analysis
of $G_t$ much more complicated. Note that for semigroups there are
general tools which allow to obtain dimension-free estimates, see
e.g.\ \cite{Ste1} and \cite[Section~4.1]{BMSW3}, \cite[Theorem~3.3]{JR},
\cite[Theorem~1.2]{MSZ1}.

In order to understand the Fourier transform estimates of the kernel $g_t$ at the origin and at infinity, we will use two representations for $\mathcal F_{\ZZ^d}g_t$. By definition the first representation has the following form
\begin{equation}
\label{eq:ft1}
\mathcal F_{\ZZ^d}{g}_t(\xi)=\frac{1}{\Theta_{1/t}(0)}\sum_{n\in \ZZ^d}\exp\bigg(- \frac{\pi|n|^2}{t}\bigg)
e(- n\cdot \xi),\qquad \xi \in \TT^d,
\end{equation}
where $e(z):=e^{2\pi i z}$ for every $z\in\CC$.
Formula \eqref{eq:ft1}  is particularly helpful when dealing with small scales,
specifically $0 < t \le 16$. This is because the exponential
$e^{-\pi/t}$ becomes very small for small values of $t$ and the series
in equation \eqref{eq:ft1} converges quickly for these values of $t$.

Considering $\varphi(y):=\exp(-\pi t |y|^2)$, $y\in \RR^d$, and 
applying the Poisson summation formula we obtain
\[
\sum_{n\in \ZZ^d}\varphi(n-\xi)
=
\sum_{n\in \ZZ^d}\mathcal F_{\RR^d} \varphi (n)
e(- n\cdot \xi),
\]
which is equivalent to the following formula
\begin{equation}
\label{eq:ps1}
\sum_{n\in \ZZ^d}\exp\bigg(- \frac{\pi|n|^2}{t}\bigg) e(- n\cdot \xi)
=
t^{d/2}\sum_{n\in \ZZ^d}\exp\big(- \pi|n-\xi|^2 t\big)
=
t^{d/2}\Theta_t(\xi).
\end{equation}
In particular, taking $\xi=0$ in the above formula we have 
\begin{equation}
\label{eq:ps2}
\Theta_{1/t}(0)
=
\sum_{n\in \ZZ^d}\exp\bigg(- \frac{\pi|n|^2}{t}\bigg)
=
t^{d/2}\sum_{n\in \ZZ^d}\exp\big(- \pi|n|^2 t\big)=t^{d/2}\Theta_t(0).
\end{equation}
Hence, coming back to \eqref{eq:ft1} we obtain the second representation given by the formula
\begin{equation}
\label{eq:ft2}
\mathcal F_{\ZZ^d}{g}_t(\xi)
=
\frac{\Theta_t(\xi)}{\Theta_t(0)},
\qquad \xi \in \TT^d,
\end{equation} 
which obviously tensorizes as
\begin{equation}
	\label{eq:ft3}
	\mathcal F_{\ZZ^d}{g}_t(\xi)=\prod_{j\in[d]}\frac{\theta_t(\xi_j)}{\theta_t(0)},\qquad \xi=(\xi_1,\ldots, \xi_d)\in\TT^d.
\end{equation}
Formula \eqref{eq:ft2}, in contrast to \eqref{eq:ft1}, is useful
when dealing with large scales, let us say $t \ge 16$. This is because
the factor $e^{-\pi t}$ is very small and the series defining
$\Theta_t$, see \eqref{eq:Theta_def}, converges very quickly for these
values of $t$.  Therefore in the proof of Theorem~\ref{thm:main}, as
well when showing various estimates for
$\mathcal F_{\ZZ^d}{g}_t(\xi)$, we naturally split our investigations
into two cases of large scales ($t \in [16,\infty)$), see
Section~\ref{sec:Festlarg}, and of small scales ($t \in (0,16)$),
see Section~\ref{sec:Festsmall}.  Despite the fact that we will
primarily use formula \eqref{eq:ft1} for small scales and formula
\eqref{eq:ft2} for large scales, it is worth noting that both formulas
\eqref{eq:ft1} and \eqref{eq:ft2} are useful for the whole region
$t>0$.
For instance, from \eqref{eq:ft2} it is not immediately clear
that $|\mathcal F_{\ZZ^d}{g}_t(\xi)|\leq 1,$ which is obvious from
\eqref{eq:ft1}. On the other hand \eqref{eq:ft2} clearly shows that
$\mathcal F_{\ZZ^d}{g}_t(\xi)\ge 0,$ contrary to \eqref{eq:ft1}.

It should be noted that even though $G_t$ is not a semigroup, the
Fourier transform $\mathcal F_{\ZZ^d}{g}_t(\xi)$ is closely related to the
heat kernel on $\TT^d$. More precisely,
\begin{equation}
\label{HKT}
H_t(\xi)
:= 
(4\pi t)^{-d/2} \Theta_{1/(4\pi t)} (\xi)
=
(4\pi t)^{-d/2}\sum_{n\in \ZZ^d}\exp\bigg(- \frac{|n-\xi|^2}{4t}\bigg),
\qquad t>0,\quad \xi\in\TT^d,
\end{equation} 
is the heat kernel on $\TT^d$.  Our estimates for
$\mathcal F_{\ZZ^d}{g}_t(\xi)$ as outlined in
Propositions~\ref{pro:est_inf} and \ref{pro:1} appear to be new when
considering the dependence on the dimension $d\in\ZZ_+$. These
estimates also directly lead to what are presumably novel estimates
for the heat kernel $H_t(\xi)$. We present the following result, which
will be justified in Appendix \ref{sec:app}.

\begin{proposition}
	\label{pro:hTd}
	For each $t\in(0,2^{-8})$ we have
	\begin{align} \label{HKT1}
	H_t(0) \exp\bigg(-\frac{|\xi|^2}{4t}\bigg)  
	\le 
	H_t(\xi) \le H_t(0) \exp\bigg(-\frac{|\xi|^2}{40t}\bigg),
	\qquad \xi \in \TT^d.
\end{align}
	Additionally,  for $t \ge 2^{-8}$ and $\xi \in \TT^d$ we have
	\begin{align} \label{HKT2}
		\begin{split}
& H_t(0) \exp\bigg(- 256 \Big( \pi^3 + 1024 \pi e^{-\pi/(64)} \Big)  e^{-4\pi^2 t}|\xi|^2\bigg) \\
& \qquad \qquad \qquad \qquad 	\le 
	H_t(\xi) 
	\le 
	H_t(0) \exp\bigg(- \frac{16}{1 + 4\sqrt{2}} e^{-4\pi^2 t}|\xi|^2\bigg).
	\end{split}
\end{align}
\end{proposition}
\noindent We remark that it is straightforward to obtain the above estimates with a multiplicative factor in front of the form $C^d H_t(0)$ for some $C = C(d)>1$. The fact that we are able to replace this factor by $H_t(0)$ may be of particular interest when the dimension $d$ tends to infinity.

\medskip
\noindent{\textbf{Acknowledgments}.}
The authors are grateful to Adam Nowak for pointing out the reference~\cite{Bod}.
The authors would like to thank the anonymous referee for carefully reading the manuscript and for providing useful comments that improved the presentation of the paper.

\section{Notation}\label{sec:not}
We now set up notation  that will be used throughout the paper.
\subsection{Basic notation}
The set of nonnegative integers and positive integers will be denoted
respectively by $\NN:=\{0,1,2,\ldots\}$ and
$ \ZZ_+:=\{1, 2, \ldots\}$.  For $d\in \ZZ_+$ the sets $\ZZ^d$,
$\RR^d$, $\CC^d$ and $\TT^d:=\RR^d/\ZZ^d$ have standard meaning.  For
any $x\in\RR$ we define the floor function
$\lfloor x \rfloor: = \max\{ n \in \ZZ : n \le x \}$ and 
the fractional part 
$\{x\}:=x-\lfloor x\rfloor$ as well as the distance to the nearest integer
by $\|x\|:={\rm dist}(x, \ZZ)$. For $x, y\in\RR$ we shall also write
$x \vee y := \max\{x,y\}$ and $x \wedge y := \min\{x,y\}$.

For
any real number $N>0$ we define
\[
[N]:=\ZZ_+\cap(0, N]=\{1, 2,\ldots, \lfloor N \rfloor\},
\]
and we will also write
\begin{align*}
\NN_{\le N}:= [0, N]\cap\NN,\ \: \quad &\text{ and } \quad
\NN_{< N}:= [0, N)\cap\NN,\\
\NN_{\ge N}:= [N, \infty)\cap\NN, \quad &\text{ and } \quad
\NN_{> N}:= (N, \infty)\cap\NN.
\end{align*}

The letter $d\in \ZZ_{+}$ is reserved for the dimension throughout this
paper and all implied constants, unless otherwise specified, will be
independent of the dimension.  For two nonnegative quantities $A, B$
we write $A \lesssim_{\delta} B$ if there is an absolute constant
$C_{\delta}>0$ depending on $\delta>0$ such that $A\le C_{\delta}B$ .
We will write $A \simeq_{\delta} B$ when
$A \lesssim_{\delta} B\lesssim_{\delta} A$. We will omit the subscript
$\delta$ if irrelevant.

We use $\ind{A}$ to denote the indicator function of a set $A$. If $S$ 
is a statement we write $\ind{S}$ to denote its indicator, equal to $1$
if $S$ is true and $0$ if $S$ is false. For instance $\ind{x\in A} = \ind{A}(x)$.

\subsection{Euclidean spaces}
The standard inner product on $\RR^d$ is denoted by
$x\cdot\xi:=\sum_{k\in [d]}x_k\xi_k$ for every $x=(x_1,\ldots, x_d) \in\RR^d$
and $\xi=(\xi_1, \ldots, \xi_d)\in\RR^d$. The inner product induces
the Euclidean norm $|x|_2:=\sqrt{x\cdot x}$, which will be abbreviated
to $|x|$. We will also consider $\RR^d$ with the following norms
\begin{align}
\label{eq:1}
|x|_q:=
\begin{cases}
\big(\sum_{k\in[d]}|x_k|^q\big)^{1/q}& \text{ if } \quad q\in[1, \infty),\\
\ \max_{k\in[d]}|x_k| & \text{ if } \quad q=\infty.
\end{cases}
\end{align}
One of the most classical examples of convex symmetric bodies are the \(q\)-balls 
\(B^q (d) \subseteq \mathbb{R}^d\), \(q \in [1, \infty]\), defined  by
\begin{align}
	\label{eq:1a}
B^q (d) :=
\{
x \in \mathbb{R}^d : |x|_q 
\leq 1
\}.
\end{align}

The $d$-dimensional torus $\TT^d:=\RR^d/\ZZ^d$ is a priori endowed
with the periodic norm
\[
\|\xi\|:=\Big(\sum_{k=1}^d \|\xi_k\|^2\Big)^{1/2}
\qquad \text{for}\qquad
\xi=(\xi_1,\ldots,\xi_d)\in\TT^d,
\]
where $\|\xi_k\|={\rm dist}(\xi_k, \ZZ)$ for all $\xi_k\in\TT$ and
$k\in[d]$.  Throughout the paper we will identify $\TT^d$ with
$[-1/2, 1/2)^d$, and we will abbreviate $\|\cdot\|$ to $|\cdot|$,
since the norm $\|\cdot\|$ coincides with the Euclidean norm $|\cdot|$
restricted to $[-1/2, 1/2)^d$.

\subsection{Function spaces and derivatives} All vector spaces in this paper will be
defined over the complex numbers $\CC$. The triple
$(X, \mathcal B(X), \mu)$ is a measure space $X$ with $\sigma$-algebra
$\mathcal B(X)$ and $\sigma$-finite measure $\mu$.  The space of all
measurable functions whose modulus is integrable with $p$-th power is
denoted by $L^p(X)$ for $p\in(0, \infty)$, whereas $L^{\infty}(X)$
denotes the space of all essentially bounded measurable functions.
The Lorentz space $L^{p, \infty}(X)$  (or weak $L^{p}(X)$ space) for $0<p<\infty$ is the set of all
measurable functions $f:X\to \CC$ such that
\begin{align}
\label{eq:2}
\|f\|_{L^{p, \infty}(X)}=\inf\{C>0: \mu(\{x\in X: |f(x)|\ge \lambda\})\le C^p\lambda^{-p} \ \text{ for all } \lambda>0\}<\infty.
\end{align}
The quantity from \eqref{eq:2} defines a quasi-norm on
$L^{p, \infty}(X)$. If $p=\infty$ we define
$L^{\infty, \infty}(X)=L^{\infty}(X)$.  In our case we will usually
have $X=\RR^d$ or $X=\TT^d$ equipped with the Lebesgue measure, and
$X=\ZZ^d$ endowed with the counting measure. If $X$ is endowed with
counting measure we will abbreviate $L^p(X)$ and $L^{p, \infty}(X)$ to
$\ell^p(X)$ and $\ell^{p,\infty}(X)$ respectively. We will also
abbreviate $\|\cdot\|_{\ell^p(\ZZ^d)}$ to $\|\cdot\|_{\ell^p}$.

For $T \colon B_1 \to B_2$, a continuous linear  map between two normed
vector spaces $B_1$ and  $B_2$, we use $\|T\|_{B_1 \to B_2}$ to denote its
operator norm. This will be usually used with $B_1=B_2=\ell^p(\ZZ^d)$.

The partial derivative of a function $f \colon \RR^d\to\CC$ with respect to
the $j$-th variable $x_j$ will be denoted by
$\partial_{x_j}f=\partial_j f$, while the $m$-fold partial derivative
will be denoted by
$\partial_{x_j}^mf=\partial_j^m f$. Let $U, V\subseteq \RR$ be open intervals and let $f \colon V\to \RR$  and $g \colon U\to V$ be two   smooth functions. We will use the Fa\'a di
Bruno's formula for the $n$-th derivative of $f\circ g$, which reads as follows
\begin{align}
\label{eq:3}
(f\circ g)^{(n)}(x)
=
\sum_{(m_1,\ldots, m_{n})\in S(n)} \frac{n!}{m_1! \cdots m_{n}!} 
f^{(m_1+\cdots + m_{n})} (g(x))\prod_{j\in[n]}\bigg(\frac{ g^{(j)}(x)}{j!}\bigg)^{m_j},
\end{align}
where $S(n):=\{(m_1,\ldots, m_{n})\in\NN^n: \sum_{j\in[n]} j m_j=n\}$.

For any open interval $U\subseteq \RR$ and any smooth function
$f \colon U \to \mathbb{C}$ such that $f'(t) \ne 0$ for all $t \in U$,
we will also use the formula \cite[(21)]{Bod} for the higher order
derivatives of the inverse function
\begin{align} \label{id:104}
(f^{-1})^{(n)} \big(f(t)\big)
=
\frac{(-1)^{n-1}}{(f'(t))^{2n-1}}
\sum_{s \in T(n)} 
\frac{(-1)^{s_1} (2n-s_1 - 2)!}{s_2 ! \cdot \ldots \cdot s_n !}
\prod_{j=1}^n \Big(\frac{f^{(j)}(t)}{j!}\Big)^{s_j}, \qquad n \in \mathbb{Z}_+,
\end{align}
where
$T(n) := \{s=(s_1,\ldots,s_n) \in \mathbb{N}^n : \sum_{j\in[n]} s_j = n-1 \, \text{
and } \, \sum_{j\in[n]} j s_j = 2(n-1)\}$. For $n=1$, we understand
that the sum in \eqref{id:104} is constantly equal to $1$.  Further,
for any $n \ge 2$ note that $T(n)\neq\emptyset$.

\subsection{Fourier transform}  
We will write $e(z):=e^{2\pi i z}$ for every $z\in\CC$.
The Fourier transform and the inverse Fourier transform of $f\in L^1(\RR^d)$ will be denoted respectively by
\begin{align*}
\mathcal F_{\RR^d} f(\xi) &:= \int_{\RR^d} f(x) e(-x\cdot \xi)\ dx,\qquad \xi\in\RR^d,\\
 \mathcal F_{\RR^d}^{-1} f(x) &:= \int_{\RR^d} f(\xi) e(x\cdot \xi)\ d\xi,\qquad x\in\RR^d.
\end{align*}
Similarly, the Fourier series of $f\in \ell^1(\ZZ^d)$, 
and the Fourier coefficient of $g\in L^1(\TT^d)$  will be denoted respectively by
\begin{align*}
	\mathcal F_{\ZZ^d} f(\xi) &:= 
	\sum_{n\in\ZZ^d} f(n) e(n\cdot \xi),\qquad \xi\in\TT^d,\\
	\mathcal F_{\ZZ^d}^{-1} g(n) &:= 
	\int_{\TT^d} g(\xi) e(-n\cdot \xi) \, d\xi,\qquad n\in\ZZ^d.
\end{align*}

 \subsection{Jumps, $r$-variations and $r$-oscillations} \label{sec:var}
 For any $\mathbb I\subseteq (0,\infty)$, any family
 $(\mathfrak a_t: t\in\mathbb I)\subseteq \CC$, and any exponent
 $1 \leq r < \infty$, the $r$-variation semi-norm is defined to be
\begin{align*}
V^{r}(\mathfrak a_t: t\in\mathbb I):=
\sup_{J\in\ZZ_+} \sup_{\substack{t_{0}<\dotsb<t_{J}\\ t_{j}\in\mathbb I}}
\Big(\sum_{j=0}^{J-1}  |\mathfrak a_{t_{j+1}}-\mathfrak a_{t_{j}}|^{r} \Big)^{1/r},
\end{align*}
where the latter supremum is taken over all finite increasing
sequences in $\mathbb I$. For any $t_0\in \mathbb I$ one has
\begin{align}
\label{eq:688}
\sup_{t\in \mathbb I}|\mathfrak{a}_t|\le |\mathfrak{a}_{t_0}|+V^r(\mathfrak{a}_t: t\in \mathbb I).
\end{align} 
 
The $r$-variation is closely related to the $\lambda$-jump counting
function. Recall that for any $\lambda>0$ the $\lambda$-jump counting
function of a family $(\mathfrak{a}_t: t\in\mathbb I)\subseteq \CC$ is
defined by
\begin{align*}
	N_\lambda(\mathfrak{a}_t : t\in \mathbb I):=\sup \{ J\in\NN : \exists_{\substack{t_{0}<\ldots<t_{J}\\ t_{j}\in\mathbb I}}  : \min_{0 \le j \le J-1} |\mathfrak{a}_{t_{j+1}} - \mathfrak{a}_{t_{j}}| \ge \lambda \}.
\end{align*}

Let now $(\mathfrak{a}_t(x): t\in \mathbb I)$ be a family of measurable functions on a $\sigma$-finite measure space $(X,\calB(X),\mu)$. Let $\mathbb I \subseteq (0,\infty)$ and $\#\mathbb I\geq2$, then for every $p\in[1,\infty]$ and $r\in[1,\infty)$ we have 
\begin{align}
\label{eq:725}
\sup_{\lambda>0}\|\lambda N_{\lambda}(\mathfrak{a}_t: t\in\mathbb I)^{1/r}\|_{L^p(X)}\le \|V^r(\mathfrak{a}_t: t\in \mathbb I)\|_{L^p(X)},
\end{align}
since for all $\lambda>0$ we have the following simple pointwise estimate
\[
\lambda N_{\lambda}(\mathfrak{a}_t(x): t\in \mathbb I)^{1/r}\le V^r(\mathfrak{a}_t(x): t\in \mathbb I), 
\qquad x \in X.
\]
It is interesting that inequality \eqref{eq:725} can be inversed in the following sense.
Let us fix $p \in [1,\infty]$ and $r>2$. Then for every family of measurable functions
$(\mathfrak{a}_t(x): t\in \mathbb I)$ 
we have the estimate
\begin{align} \label{estt1}
	\big\| V^{r}\big( \mathfrak{a}_t : t \in \mathbb I \big) \big\|_{L^{p,\infty} (X)}
	\lesssim_{p,r}
	\sup_{\lambda>0} \big\| \lambda N_{\lambda}(\mathfrak{a}_t: t\in \mathbb I)^{1/2} \big\|_{L^{p,\8}(X)}.
\end{align}
 A proof of \eqref{estt1} can be found in \cite[Lemma 2.3, p.\ 805]{MSZ1}. This estimate says that $\lambda$-jumps are an endpoint for $r$-variations.
 
 Let
 $\mathbb I\subseteq (0,\infty)$ be  such that $\# \II\ge2$ and for
 every $J\in\ZZ_+$ define the set
 \begin{align*}
	\mathfrak S_J(\II):
=
\big\{ (I_{0}, \ldots, I_{J}) \subseteq \II^{J+1} : 
I_{0} < I_{1} < \ldots < I_{J} \big\}.
 \end{align*}
 In other
 words, $\mathfrak S_J(\II)$ is the family of all strictly increasing
 sequences of length $J+1$
 taking their values in the set $\II$.
 
 For any $\mathbb I\subseteq (0,\infty)$, any family
 $(\mathfrak a_t: t\in\mathbb I)\subseteq \CC$, any exponent
 $1 \leq r < \infty$, any $J\in\ZZ_+$ and a sequence
 $I \in \mathfrak S_J(\II)$, the $r$-oscillation semi-norm is defined to be
\begin{align*}
	O_{I, J}^r(\mathfrak a_{t} : t \in \II):=
	\Big(\sum_{j=0}^{J-1}\sup_{t\in [I_j, I_{j+1}) \cap\II}
	\abs{\mathfrak a_{t} - \mathfrak a_{I_j} }^r\Big)^{1/r}.
\end{align*}
Here we use the
convention that the supremum taken over the empty set is zero.  Note that
for any $\II\subseteq (0,\infty)$,
any $r\in [1,\infty)$, any family $(\mathfrak a_t:t\in\II)\subseteq \CC$, any  $J\in\ZZ_+$, and any  $I\in\mathfrak S_J(\II)$  one has
\begin{align}
	\label{eq:62}
	O_{I, J}^r(\mathfrak a_{t}: t \in \II)\le V^r(\mathfrak{a}_t: t\in \II)\le 2\Big(\sum_{t\in\II}|\mathfrak a_{t}|^r\Big)^{1/r},
\end{align}
the latter object being usually $\infty$ if $\II$ is uncountable.
Consequently, if now $\mathbb I \subseteq (0,\infty)$ with $\#\mathbb I\geq2$, and 
$(\mathfrak{a}_t(x): t\in \mathbb I)$ is a family of measurable functions on a $\sigma$-finite measure space $(X,\calB(X),\mu)$, then for every $p\in[1,\infty]$ and $r\in[1,\infty)$ we have 
\begin{align}
	\label{eq:7255}
	\sup_{J\in\ZZ_+}\sup_{I\in \mathfrak{S}_J(\II)} 	
	\big\| O_{I, J}^{r}
	(\mathfrak{a}_t: t\in\mathbb I)\|_{L^p(X)}\le \|V^r(\mathfrak{a}_t: t\in \mathbb I)\|_{L^p(X)}.
\end{align}
Despite the similarity of \eqref{eq:725} and \eqref{eq:7255}, it was proved in \cite[Theorem~1.2]{MiSlSz} that $2$-oscillations cannot be interpreted as an endpoint for $r$-variations
when $r > 2$, in the sense of inequality \eqref{estt1}. Actually much weaker estimate than an analogue of \eqref{estt1} fails to be true, see \cite[Theorem 1.2]{MiSlSz}.

 For further details related to $\lambda$-jump function, $r$-variation and $r$-oscillation semi-norms we refer the interested reader to a survey article \cite{MSW}.

\section{Fourier transform estimates at the origin} \label{sec:Fouest}

We begin with straightforward estimates for the one-dimensional multiplier $\theta_t.$
\begin{lemma}
\label{lem:1}
For all $t>0$ we have
\[
\max\big(1,t^{-1/2}\big)\le \theta_t(0)\le 1+t^{-1/2}.
\]
\end{lemma}
\begin{proof}
A simple comparison shows that
\[
\sum_{n\in\ZZ_+}e^{-\pi t n^2}\le \int_0^{\infty}e^{-\pi t x^2}\,dx=\frac{1}{2\sqrt{t}}.
\]
Using $\theta_t(0)=1+2\sum_{n\in\ZZ_+}e^{-\pi t n^2}$ we obtain
 the desired upper bound $\theta_t(0)\le 1+t^{-1/2}$.
Further, it is clear that $\theta_t(0)\ge 1$. Next, by \eqref{eq:ps2}
for $d=1$, we also have
$\theta_t(0)=t^{-1/2}\theta_{1/t}(0)\ge t^{-1/2}$. Combining these two
estimates we obtain the desired lower bound and the proof of
Lemma~\ref{lem:1} is finished.
\end{proof}

For $x>0$ we have two simple  inequalities 
\begin{align} \label{id:101}
xe^{-x} \le 
e^{-x/2}, \qquad \text{ and }\qquad
\frac{x e^{-x}}{1-e^{-x}} \le  e^{-x/2}.
\end{align}

The main result of this section is the following estimate for
$\mathcal F_{\ZZ^d}g_t$ at the origin.

\begin{proposition}
\label{pro:est_0}
For each $d\in \ZZ_{+}$, $\xi\in \TT^d$ and $t>0$ we have
\begin{align}
\label{id:105}
\big|\mathcal F_{\ZZ^d}g_t(\xi)-1\big| \le 6 \pi t |\xi|^2.
\end{align}	
Further, for any fixed $D > 0$, we have a more precise estimate
\begin{align}
\label{id:106}
\big|\mathcal F_{\ZZ^d}g_t(\xi)-1\big| \lesssim_D e^{-\pi/t} |\xi|^2,
\end{align}	
which holds uniformly in $d\in \ZZ_{+}$, $\xi\in \TT^d$ and $t \le D$ with
the implicit constant depending only on $D$.

\end{proposition}

\begin{proof}
Using the formula \eqref{eq:ft1} for $\mathcal F_{\ZZ^d}g_t$ and symmetry we write
 \[
 \mathcal F_{\ZZ^d}g_t(\xi) = \frac{1}{\Theta_{1/t}(0)}
 \sum_{n\in \ZZ^d}\exp\big(- \pi|n|^2{t^{-1}}\big)\prod_{k\in[d]} \cos (2\pi n_k \xi_k),
 \]
 so that
 \[
 \mathcal F_{\ZZ^d}g_t(\xi)-1 = \frac{1}{\Theta_{1/t}(0)}\sum_{n\in \ZZ^d}\exp\big(- \pi|n|^2{t^{-1}}\big)\Big[ \Big(\prod_{k\in[d]} \cos (2\pi n_k \xi_k)\Big)-1\Big].
 \]
 Now taking into account the inequalities
 \begin{align*}
 \Big|\prod_{k\in[d]} a_k -1\Big|
 &\leq
 \sum_{k\in [d]} |a_k-1|, \qquad \text{for } \quad |a_k|\le 1, \\
 |\sin x|&\le |x|, \qquad \text{for } \quad x \in \mathbb{R},
 \end{align*} 
and the formula $\cos 2x=1-2\sin^2 x$ together with the product structure  of $\Theta_{1/t}(0)$, see \eqref{eq:Theta_def},  we obtain
 \begin{align} \label{id:107}
 \begin{split}
 \left| \mathcal F_{\ZZ^d}g_t(\xi)-1\right| &\le \frac{2}{\Theta_{1/t}(0)} \sum_{n\in \ZZ^d}
\exp\big(- \pi|n|^2{t^{-1}}\big)\sum_{k\in[d]} \sin^2 (\pi n_k \xi_k) \\
 & \le
 \frac{2\pi^2}{\Theta_{1/t}(0)}\sum_{n\in \ZZ^d}\exp\big(- \pi|n|^2{t^{-1}}\big)\sum_{k\in[d]} n_k^2 \xi_k^2\\
 &=\frac{2\pi^2}{\Theta_{1/t}(0)}\sum_{k\in [d]} \xi_k^2 \sum_{n\in \ZZ^d}n_k^2 \exp\big(- \pi|n|^2{t^{-1}}\big)\\
 &= 2\pi^2 |\xi|^2 \theta_{1/t}(0)^{-1}\sum_{j\in \ZZ} j^2\exp\big(- \pi j^2 {t^{-1}}\big).
\end{split}
 \end{align}
The first inequality from \eqref{id:101} and Lemma~\ref{lem:1} imply
\begin{align*} 
\left|\mathcal F_{\ZZ^d}g_t (\xi) - 1\right|
&\leq 
2 \pi t |\xi|^2 \theta_{1/t}(0)^{-1}\sum_{j\in \ZZ} (\pi j^2t^{-1})\exp\big(- \pi j^2 {t^{-1}}\big)\\
&\le 
2 \pi t |\xi|^2 \frac{\theta_{1/(2t)}(0)}{\theta_{1/t}(0)}\\
&\le 2 \pi t|\xi|^2 \frac{1+(2t)^{1/2}}{\max(1,t^{1/2})}\\
&\le 6 \pi t |\xi|^2,
\end{align*}
which proves \eqref{id:105}. On the other hand, using once again \eqref{id:101} we obtain
\begin{align*}
\sum_{j\in \ZZ} j^2 \exp\big(- \pi j^2 {t^{-1}}\big)
\le	
2 \exp\big(- \pi {t^{-1}}\big)
+ 
2 \pi^{-1} t \sum_{j\ge 2}  \exp\big(- \pi j {t^{-1}}\big)
\simeq_D
\exp\big(- \pi {t^{-1}}\big), \qquad t \le D.
\end{align*}
This combined with \eqref{id:107} and Lemma~\ref{lem:1} yields \eqref{id:106} and 
the proof is completed.
\end{proof}

\section{Fourier transform estimates at infinity: large scales} \label{sec:Festlarg}

\begin{lemma}
\label{lem:2}
For $t>0$ and $\xi\in \TT$ we have
\begin{equation*}
\theta_t(\xi)\le \theta_{t/8}(0)\exp\big(-\pi t|\xi|^2/2\big)
\le 
\big(1+3t^{-1/2}\big)\exp\big(-\pi t|\xi|^2/2\big).
\end{equation*}
\end{lemma}
\begin{proof}
The second inequality follows directly from Lemma \ref{lem:1}, thus we
only focus on proving the first inequality. Considering separately
$n=0$ and $|n|\ge 1$ it is easy to see that for $\xi \in \TT$ we
have $|n-\xi|^2 \ge n^{2}/8+ |\xi|^{2}/2$. Indeed, the case $n=0$ is trivial and for $|n|\ge 1$ we have
\[
|n-\xi|^2=|n/2+(n/2-\xi)|^2\ge n^{2}/4 \ge n^{2}/8+ |\xi|^{2}/2,
\]
since $|\xi|\le 1/2$.  Therefore, by definition of $\theta_{t/8}$ we have 
 \begin{align*}
 \theta_t(\xi)
 \le 
 \exp\big(-\pi t|\xi|^2/2\big) 
 \sum_{n\in \ZZ} \exp\big(-\pi t n^2/8\big)
 = 
 \theta_{t/8}(0)\exp\big(-\pi t|\xi|^2/2\big),
 \end{align*}
 which completes the
proof of the Lemma~\ref{lem:2}.
\end{proof}

If we consider the $d$-dimensional multiplier $\Theta_t(\xi)$ for
$\xi \in \TT^d$, then 
Lemma \ref{lem:2} yields
\begin{equation}
\label{eq:Tedim1}
\mathcal F_{\ZZ^d}g_t(\xi)\le \frac{\Theta_{t/8}(0)}{\Theta_t(0)} 
\exp \big(-\pi t|\xi|^2/2\big),
\qquad t>0, \quad \xi \in \TT^d.
\end{equation}
If we  use the second inequality from Lemma \ref{lem:2}, then we can further bound \eqref{eq:Tedim1} as follows
\begin{equation}
\label{eq:Tedim1'}
\mathcal F_{\ZZ^d}g_t(\xi)\le \big(1+3t^{-1/2}\big)^d
\exp \big(-\pi t|\xi|^2/2\big),
\qquad t>0, \quad \xi \in \TT^d.
\end{equation}
This estimate is not sufficient for our purposes because the implied
constant depends on the dimension even when $t>0$ is 
large, up to $t\simeq d^2$. However, we can replace the implied constant in equation \eqref{eq:Tedim1'}
with $1$ whenever $t\ge 15$ while still maintaining an exponential
decay. More specifically, we prove the following.

\begin{proposition}
\label{pro:est_inf}
For $d \in \ZZ_{+}$, $t\ge 15$ and $\xi\in\TT^d$ we have
\[
\mathcal F_{\ZZ^d}g_t(\xi) \le \exp\big(-\pi t |\xi|^2/10 \big).
\]
\end{proposition}

The next lemma will be key to prove Proposition \ref{pro:est_inf}.

\begin{lemma}
\label{lem:4}
Fix $0<\alpha<1/2$. Then for all $t>0$ and $\xi\in \TT$ such that
$|\xi|\le \alpha$ we have
\begin{equation}
\label{eq:lem:4:1}
\exp\big(-\pi t \xi^2\big)
\le \frac{\theta_t(\xi)}{\theta_t(0)}
\le \exp\big(-\pi t \xi^2\big) \exp\bigg( \frac{2 \xi^2 e^{-\pi t (1/4-\alpha/2)} }{(1/2-\alpha)^2}\bigg).
\end{equation}
Further, the lower bound holds true for any $t>0$ and $\xi\in \TT$.
\end{lemma}

\begin{proof}
Here, it is important that $\theta_t(\xi)$ is even in $\xi$.
Define $a_0=1$ and $a_n=2$ for $n\in\ZZ_+$. Then, for $\xi\in\TT$  we have
\begin{equation}
\label{eq:lem:4:form}
\theta_t(\xi)
=
\Big(\sum_{n\in\ZZ} e^{-\pi tn^2} e^{2\pi tn\xi}\Big)e^{-\pi t\xi^2}
=
\Big(\sum_{n\in\NN} a_n e^{-\pi tn^2} \cosh(2\pi tn\xi) \Big)e^{-\pi t\xi^2},
\end{equation}
so that
\begin{equation}
\label{eq:lem:4:2}
\frac{\theta_t(\xi)}{\theta_t(0)}=e^{-\pi t\xi^2}  I_t(\xi)	
\end{equation}
with
\begin{align*}
I_t(\xi):=\theta_t(0)^{-1}\sum_{n\in \NN} a_n e^{-\pi tn^2} \cosh(2\pi tn\xi).
\end{align*}
Thus it remains to estimate $I_t(\xi).$ 
By the formula $\theta_t(0)=\sum_{n\in\NN}a_ne^{-\pi t n^2}$ and since $\theta_t(0) \ge 1$, we obtain 
\[
0\le I_t(\xi)-1=\theta_t(0)^{-1}\sum_{n\in\ZZ_+} a_n e^{-\pi tn^2} [\cosh(2\pi tn\xi)-1]
\le 4\sum_{n\in\ZZ_+}  e^{-\pi tn^2}\sinh^2(\pi t n\xi),
\]
for $t> 0$ and $\xi \in \TT$. In the last inequality above we used the fact that
 $\cosh(2\pi tn\xi)-1=2\sinh^2(\pi t n\xi)$.
This, in particular, shows the first estimate in \eqref{eq:lem:4:1}.

Further, using the inequalities 
\[
|\sinh x| \le |x| \cosh x, \qquad \text{ and } \qquad
\cosh x \le e^{|x|}, \qquad x \in \mathbb{R}, 
\]
we may conclude, for $|\xi|\le \alpha$, that
\begin{align*}
|I_t(\xi)-1|
&\le 
4 \pi^2 t^2 \xi^2 \sum_{n\in\ZZ_+}n^2 e^{-\pi t n^2}\cosh^2(\pi t n \alpha)\\
&\le 
4 \pi t \xi^2\sum_{n\in\ZZ_+}\pi t n^2e^{-\pi t n^2}e^{2\pi t n \alpha} \\
& \le 
4 \pi t \xi^2\sum_{n\in\ZZ_+}\pi t n^2e^{-\pi t n^2(1-2\alpha)}.
\end{align*}
By both the inequalities in \eqref{id:101}
we obtain
\begin{align*}
|I_t(\xi)-1|
& \le 
\frac{4 \pi t \xi^2}{1-2 \alpha} \sum_{n\in\ZZ_+}e^{-\pi t n(1/2-\alpha)}\\
&=
\frac{4 \pi t \xi^2}{1-2 \alpha} \frac{e^{-\pi t (1/2-\alpha)}}{1-e^{-\pi t(1/2-\alpha)}}\\
& \le 
\frac{2 \xi^2}{(1/2-\alpha)^2} e^{-\pi t (1/4-\alpha/2)}.
\end{align*}
By invoking \eqref{eq:lem:4:2}, and using the inequality $1+x \le e^x$ for $x \ge 0$, 
we obtain
\begin{align*}
\frac{\theta_t(\xi)}{\theta_t(0)}
\le 
e^{-\pi t \xi^2} 
\bigg(1+\frac{2 \xi^2}{(1/2-\alpha)^2} e^{-\pi t (1/4-\alpha/2)} \bigg)
\le 
e^{-\pi t \xi^2}\exp\bigg(\frac{2 \xi^2}{(1/2-\alpha)^2} e^{-\pi t (1/4-\alpha/2)} \bigg).
	\end{align*}
This completes the proof of the lemma.
\end{proof}

As a corollary of Lemmas \ref{lem:2} and \ref{lem:4} we obtain the following useful estimate.

\begin{corollary}
\label{cor:1}
For $t\ge 15$ and $\xi \in \TT$ we have
\begin{equation}
\label{eq:cor:1:1}
\frac{\theta_t(\xi)}{\theta_t(0)}\le \exp\big(-\pi t \xi^2/10 \big).
\end{equation}
Specifically, \eqref{eq:cor:1:1} together with  \eqref{eq:ft3} yield Proposition \ref{pro:est_inf}.
\end{corollary}
\begin{proof}
Consider first $|\xi|\le 1/4.$ Note that for $t\ge 10$ we have
\begin{align*}
\frac{2 e^{-\pi t (1/4-1/8)} }{(1/2-1/4)^2}
=
32e^{-\pi t /8}
\le 
32e^{-5 \pi/4}
\le 
\frac{9}{10},
\end{align*}
so that Lemma \ref{lem:4} with $\alpha=1/4$ gives
\[
\frac{\theta_t(\xi)}{\theta_t(0)}\le \exp\big(-\pi t \xi^2/10 \big).
\]
Assume now that $|\xi|\ge 1/4.$ Then, for $t\ge 15$ an application of Lemma~\ref{lem:1} and Lemma~\ref{lem:2} gives
\[
\frac{\theta_t(\xi)}{\theta_t(0)}
\le \exp\big(1-\pi t \xi^2/2 \big)
\le \exp\big(16 \xi^2-\pi t \xi^2/2 \big)
\le \exp\big(-\pi t \xi^2/10 \big).
\]
Consequently, the proof of Corollary \ref{cor:1} is justified. 
\end{proof}

The second important ingredient is the following estimate for the derivatives of $\mathcal F_{\ZZ^d}g_t(\xi)$.

\begin{proposition}
\label{pro:est_der}
For each $n\in \NN$ we have that
\[
\big|t^n \partial_t^n \mathcal F_{\ZZ^d}g_t(\xi)\big| \lesssim_n 1,
\]
holds uniformly in $d \in \ZZ_{+}$, $t\ge 15$ and $\xi\in \TT^d$.
\end{proposition}

We begin with a one-dimensional variant of Proposition \ref{pro:est_der}, which can be stated as follows.

\begin{lemma}
\label{lem:est_der_1dim}
For each $n\in \NN$ we have that
\begin{equation}
\label{eq:est_der_1dim}
\bigg|t^n \partial_t^n\,\bigg(\frac{\theta_t(\xi)}{\theta_t(0)}\bigg)\bigg|
\lesssim_n 
\big(t^n\xi^{2n}+\xi^2 e^{-\pi t/5} \big) 
\exp\big(-\pi t |\xi|^2/5  \big),
\end{equation}
holds uniformly in $t>1$ and $\xi\in \TT$.
\end{lemma}

\begin{remark}
\label{rem:1:lem:est_der_1dim}
Now two remarks are in order.

\begin{itemize}
\item[(i)] Estimate \eqref{eq:est_der_1dim} differs from what happens for the derivatives of the Gaussian kernel on $\RR.$ Namely, 
\[t^n \partial_t^n \exp(-\pi t x^2)= (-\pi)^n t^n x^{2n}\exp(-\pi t x^2),\qquad x\in \RR,\]
so that we do not have the extra second term as in  \eqref{eq:est_der_1dim}. This limits the applicability of Lemma~\ref{lem:est_der_1dim} to only large scales  $t$. Small scales will be handled differently in the next section.

\item[(ii)] For $n=0$, Lemma \ref{lem:est_der_1dim} gives a better
exponential decay $\exp(-\pi t |\xi|^2/5)$ than in Corollary~\ref{cor:1}. 
The price is that the implicit constant in Lemma~\ref{lem:est_der_1dim} is not necessairly $1$. For this reason
Corollary~\ref{cor:1} is vital for obtaining Proposition~\ref{pro:est_inf} with a constant independent of the dimension.

\end{itemize}
\end{remark}

\begin{proof}[Proof of Lemma \ref{lem:est_der_1dim}]
The proof is technical, so we will proceed in two steps. We fix $n\in\NN$ and $t>1$.

\paragraph{\bf Step 1} Consider first the case when $|\xi|\ge 1/4$. Then
\begin{equation}
\label{eq:lem:est_der_1dim_1}
\partial_t^n\,\bigg(\frac{\theta_t(\xi)}{\theta_t(0)}\bigg)
=
\sum_{k=0}^{n-1}{n\choose k} \partial_t^k \theta_t(\xi) \partial_t^{n-k}\bigg(\frac{1}{\theta_t(0)}\bigg)
+
\frac{\partial_t^n \theta_t(\xi)}{\theta_t(0)}.	
\end{equation}
Moreover, 
\[
\partial_t^k \theta_t(\xi)=\sum_{j\in \ZZ}(-\pi (j-\xi)^2)^k e^{-\pi t (j-\xi)^2},
\]
so that using Lemma \ref{lem:2} we obtain, for $k\in\NN_{\le n}$, that 
\begin{equation}
\label{eq:lem:est_der_1dim_2}
\big|	t^k \partial_t^k \theta_t(\xi)\big|\lesssim_n \sum_{j\in \ZZ} e^{-\pi t (j-\xi)^2/2}=\theta_{t/2}(\xi)\lesssim \exp\big(-\pi t |\xi|^2/4  \big).
\end{equation}
Additionally, when $l\in \ZZ_+$, we have
\[
t^l \partial_t^l [\theta_t(0)]
=
t^l \partial_t^l\big[1+2\sum_{j\in\ZZ_{+}}e^{-\pi t j^2}\big]
=
2 \sum_{j\in \ZZ_{+} }(-\pi t j^2)^l e^{-\pi t j^2},
\]
 and thus we have 
\begin{equation}
\label{eq:lem:est_der_1dim_3} 
t^l |\partial_t^l [\theta_t(0)]|
\lesssim_l 
e^{-\pi t/2},\qquad t>1, \quad l \in\ZZ_+.
\end{equation}
Now, using Fa\'a di Bruno's formula for the $(n-k)$-th derivative
$(f\circ g)^{(n-k)}$ with $f(x)=1/x$ and $g(t)=\theta_t(0)$ we obtain
\begin{align*}
\partial_t^{n-k}\bigg(\frac{1}{\theta_t(0)}\bigg)
&=
\sum \frac{(n-k)!}{m_1! \cdot \ldots \cdot m_{n-k}!}f^{(m_1+\cdots + m_{n-k})}(\theta_t(0))\prod_{j=1}^{n-k}\bigg(\frac{\partial_t^j \theta_t(0)}{j!}\bigg)^{m_j}\\
&=
\sum \frac{(n-k)!(m_1+\cdots + m_{n-k})!(-1)^{m_1+\cdots+m_{n-k}}}{(m_1! \cdot \ldots \cdot m_{n-k}!)(\theta_t(0))^{m_1+\cdots+m_{n-k}+1}} \prod_{j=1}^{n-k}\bigg(\frac{\partial_t^j \theta_t(0)}{j!}\bigg)^{m_j},
\end{align*}
where the (finite) sums above run over all $(n-k)$-tuples of
non-negative integers $(m_1,\ldots,m_{n-k})$ satisfying the constraint
$\sum_{j=1}^{n-k} j m_j=n-k.$ In particular, if $k<n,$ then some of
the $m_j$ must be at least $1.$ Hence, using
\eqref{eq:lem:est_der_1dim_3} together with the fact that
$\theta_t(0)\ge 1$ we obtain for $k\in\NN_{<n}$ the bound
\begin{equation}
\label{eq:lem:est_der_1dim_3'} 
t^{n-k}
\left|\partial_t^{n-k}\bigg(\frac{1}{\theta_t(0)}\bigg)\right|
\lesssim_n 
e^{-\pi t /2},\qquad t>1.
\end{equation}
Combining \eqref{eq:lem:est_der_1dim_1}, \eqref{eq:lem:est_der_1dim_2}
and \eqref{eq:lem:est_der_1dim_3'} for $|\xi|\ge 1/4$ we arrive at
\begin{equation*}
\begin{split}
\left|t^n \partial_t^n\,\bigg(\frac{\theta_t(\xi)}{\theta_t(0)}\bigg)\right| &= \left|\sum_{k=0}^{n-1}{n\choose k} t^k\partial_t^k \big( \theta_t(\xi) \big) t^{n-k}\partial_t^{n-k}\bigg(\frac{1}{\theta_t(0)}\bigg) +
\frac{t^n \partial_t^n \theta_t(\xi)}{\theta_t(0)} \right|\\
&\lesssim_n \big(e^{-\pi t/2}+1\big)\exp\big(-\pi t |\xi|^2/4 \big)\\
&\simeq \exp\big(-\pi t |\xi|^2/4 \big)\\
&\lesssim_n t^n |\xi|^{2n}\exp\big(-\pi t |\xi|^2/4 \big),
		 \end{split}
\end{equation*}
where in the last inequality we used the restrictions $t>1$ and $|\xi|\ge 1/4$.

\paragraph{\bf Step 2} It remains to consider the case when $|\xi|\le 1/4$. Denote 
\[
\varphi_t(\xi)
=
2\sum_{j\in\ZZ_+}e^{-\pi t j^2} \big(\cosh (2\pi tj \xi)-1\big).
\]
Then, using \eqref{eq:lem:4:form} we see that
\begin{equation}
\label{eq:lem:est_der_1dim_5}
\frac{\theta_t(\xi)}{\theta_t(0)}
=
e^{-\pi t \xi^2}  \theta_t(0)^{-1}\big(1+2\sum_{j\in\ZZ_+}e^{-\pi t j^2} \cosh (2\pi tj \xi)\big)
=
e^{-\pi t \xi^2}\bigg(\frac{\varphi_t(\xi)}{\theta_t(0)}+1\bigg).
\end{equation}	
Next, we prove that for every $l \in\NN$ we have
\begin{equation}
\label{eq:lem:est_der_1dim_6}
|\partial^l_t \varphi_t(\xi)|
\lesssim_l 
\xi^2 e^{-\pi t/4},\qquad t>1, \quad |\xi|\le 1/4.
\end{equation} 	
To prove \eqref{eq:lem:est_der_1dim_6} we denote
\[
h(x) = \frac{\cosh 2\pi x -1}{(2\pi x)^2},\qquad x\in \RR.
\]
By a direct computation using the series expansion of $\cosh 2\pi x$ we obtain
\begin{equation}
		\label{eq:lem:est_der_1dim_7}
	 |\partial^{l}_x h(x)|
	 \le 
	 (2\pi)^{l} 
	 e^{2\pi |x|},\qquad x\in \RR,\quad l\in \NN.
 \end{equation} 
Then
\[
\varphi_t(\xi)
=
2\sum_{j\in\ZZ_+}e^{-\pi t j^2}h(t j \xi)(2\pi t j \xi)^2=8\pi ^2 \xi^2 \sum_{j\in\ZZ_+}(tj)^2 e^{-\pi t j^2}h(tj\xi),\]
so that using the Leibniz rule followed by \eqref{eq:lem:est_der_1dim_7} we obtain, for $t>1$ and $|\xi|\le 1/4,$ that
 \begin{align*}
 |\partial^l_t \varphi_t(\xi)|
 &\lesssim_l 
 \xi^2 \sum_{j\in\ZZ_+} j^{2l+2} t^2 e^{-\pi t j^2} e^{2\pi tj|\xi|}\\
& \leq 
 \xi^2 \sum_{j\in\ZZ_+} j^{2l+2} t^2 e^{-\pi t j^2} e^{\pi tj/2} \\
 & \lesssim 
 \xi^2 t^2 \sum_{j\in\ZZ_+} j^{2l+2} e^{-\pi t j/2} \\
& \lesssim_l 
 \xi^2 t^2 \sum_{j\in\ZZ_+}  e^{-\pi t j/3}\\
& \lesssim 
 \xi^2 e^{-\pi t/4}.
 \end{align*}
 Thus, \eqref{eq:lem:est_der_1dim_6} is justified.  Returning to
 \eqref{eq:lem:est_der_1dim_5} and using the Leibniz rule we have
\begin{align*}
\partial_t^n \bigg(	\frac{\theta_t(\xi)}{\theta_t(0)}\bigg) 
&=
\partial_t^n (e^{-\pi t \xi^2}) \bigg(\frac{\varphi_t(\xi)}{\theta_t(0)}+1\bigg)
+
\sum_{k=0}^{n-1} {n\choose k} \partial_t^k(e^{-\pi t \xi^2}) \partial_t^{n-k}\bigg(\frac{\varphi_t(\xi)}{\theta_t(0)}+1\bigg)\\
&= (-\pi \xi^2)^n e^{-\pi t \xi^2} \bigg(\frac{\varphi_t(\xi)}{\theta_t(0)}+1\bigg)
+
\sum_{k=0}^{n-1} {n\choose k} (-\pi \xi^2)^k e^{-\pi t \xi^2} \partial_t^{n-k}\bigg(\frac{\varphi_t(\xi)}{\theta_t(0)}\bigg).
\end{align*}
Using \eqref{eq:lem:est_der_1dim_6} with $l=0$ we see that 
\begin{equation}
\label{eq:lem:est_der_1dim_7'}
\left|t^n \partial_t^n \bigg(	\frac{\theta_t(\xi)}{\theta_t(0)}\bigg)\right|\lesssim_n t^n \xi^{2n}e^{-\pi t \xi^2}+ t^n \sum_{k=0}^{n-1}  \xi^{2k} e^{-\pi t \xi^2} \left|\partial_t^{n-k}\bigg(\frac{\varphi_t(\xi)}{\theta_t(0)}\bigg)\right|.
\end{equation}
Yet another application of the Leibniz rule gives for $k\in\NN_{<n}$ the formula
\[
\left|\partial_t^{n-k}\bigg(\frac{\varphi_t(\xi)}{\theta_t(0)}\bigg)\right| = \left|\big(\partial_t^{n-k} \varphi_t (\xi)\big)\frac{1}{\theta_t(0)}+\sum_{l=0}^{n-k-1}{n-k \choose l}\partial_t^l \big(\varphi_t (\xi)\big)\partial_t^{n-k-l}\bigg(\frac{1}{\theta_t(0)}\bigg)\right|.
\]
Since in the sum above $n-k-l\ge 1,$ using
\eqref{eq:lem:est_der_1dim_3'} and \eqref{eq:lem:est_der_1dim_6} we
obtain
\[
\left|\partial_t^{n-k}\bigg(\frac{\varphi_t(\xi)}{\theta_t(0)}\bigg)\right| \lesssim_n \xi^2 e^{-\pi t/4}+ \xi^2 e^{-3\pi t/4} \simeq \xi^2 e^{-\pi t/4}.
\]
By this estimate and \eqref{eq:lem:est_der_1dim_7'} we deduce, for $t>1$
and $|\xi|\le 1/4$, that
\begin{align*}
\left|t^n \partial_t^n \bigg(	\frac{\theta_t(\xi)}{\theta_t(0)}\bigg)\right| \lesssim_n t^n \xi^{2n} e^{-\pi t\xi^2}+t^n\xi^2 e^{-\pi t/4}e^{-\pi t \xi^2}
\lesssim_n 
\big(t^n \xi^{2n}+\xi^2 e^{-\pi t /5} \big) e^{-\pi t \xi^2}.
\end{align*}
This completes the proof in the case $|\xi|\le 1/4,$ and thus also the proof of Lemma \ref{lem:est_der_1dim}. 
\end{proof}

We are now ready to prove Proposition~\ref{pro:est_der}. Actually we shall prove a stronger statement.

\begin{proposition}
\label{pro:est_der_gauss}
For each $n\in \NN$ there exists $C_n>0$ independent of the dimension
$d\in\ZZ_{+}$ such that
\begin{align}
\label{eq:4}
\big|t^n\partial_t^n\mathcal F_{\ZZ^d}g_t(\xi)\big|
\leq C_n \exp\big(-\pi t |\xi|^22^{-n}/ 10 \big),
\end{align}
holds for all $t\ge 15$ and $\xi\in \TT^d$. 
\end{proposition}

\begin{proof}
We will prove \eqref{eq:4} proceeding by induction on $n\in\NN$. For
$n=0$ inequality \eqref{eq:4} is true thanks to
Proposition~\ref{pro:est_inf}.  
Further, observe that using Lemma~\ref{lem:est_der_1dim} we get \eqref{eq:4} for $d=1$ and any $n \in \NN$.
Therefore from now on we assume that $d>1$.
Throughout the proof for
$\xi^{(k)}:=(\xi_1,\ldots,\xi_{k-1},\xi_{k+1},\ldots,\xi_d)$ we let
\[
\mathcal F_{\ZZ^d}g_t(\xi^{(k)}):=\prod_{j\in[d]\setminus\{k\}}\frac{\theta_t(\xi_j)}{\theta_t(0)},
\]
so that, for each $k\in[d]$ we have
\[
\mathcal F_{\ZZ^d}g_t(\xi)=\frac{\theta_t(\xi_k)}{\theta_t(0)}\mathcal F_{\ZZ^d}g_t(\xi^{(k)}).
\]
Assume now that inequality \eqref{eq:4} holds for all
$j\in\NN_{\le n}$. We now prove that inequality \eqref{eq:4} is true
with $n+1$ in place of $n$. Applying the Leibniz rule twice we obtain
\begin{equation}
\label{eq:pro:est_der_gauss_1}
\begin{split}
t^{n+1}\,\partial_t^{n+1}\mathcal F_{\ZZ^d}g_t(\xi)&=t^{n+1}\, \partial_t^n\partial_t\bigg(\prod_{k\in[d]}\frac{\theta_t(\xi_k)}{\theta_t(0)}\bigg) =
t^{n+1}\, \partial_t^n \bigg(\sum_{k\in[d]}  \partial_t \bigg(\frac{\theta_t(\xi_k)}{\theta_t(0)}\bigg)\mathcal F_{\ZZ^d}g_t(\xi^{(k)})\bigg)\\
&=\sum_{k\in[d]} \sum_{j=0}^{n} {n \choose j} \left[t^{j+1}\partial_t^{j+1} \bigg(\frac{\theta_t(\xi_k)}{\theta_t(0)}\bigg)\right] \left[t^{n-j}\partial_t^{n-j}\mathcal F_{\ZZ^d}g_t(\xi^{(k)})\right].
\end{split}
\end{equation}
By Lemma \ref{lem:est_der_1dim}, we have for $t\ge 15$ and
$\xi_k\in \TT$, the following bound
\[
\left|t^{j+1}\partial_t^{j+1} \bigg(\frac{\theta_t(\xi_k)}{\theta_t(0)}\bigg)\right|
\lesssim_j 
\big(t^{j+1}|\xi_k|^{2(j+1)}+\xi_k^2 e^{-\pi t/5}\big) 
\exp\big(-\pi t \xi_k^2/5 \big).
\]
By the induction hypothesis in dimension $d-1$ with $\xi^{(k)}$ in
place of $\xi$ applied for $n-j$ one has
\[
\big|t^{n-j}\partial_t^{n-j}\mathcal F_{\ZZ^d}g_t(\xi^{(k)})\big|
\lesssim_{n} 
\exp\big(-\pi t |\xi^{(k)}|^2 2^{-n}/10  \big). 
\]
By plugging these two bounds into identity \eqref{eq:pro:est_der_gauss_1}, we obtain that
\[
\left|t^{n+1}\,\partial_t^{n+1}\mathcal F_{\ZZ^d}g_t(\xi)\right| \lesssim_n \sum_{j=0}^n \sum_{k\in[d]} \big(t^{j+1}|\xi_k|^{2(j+1)}+\xi_k^2 e^{-\pi t/5}\big) \exp\big(-\pi t |\xi|^22^{-n}/ 10 \big).
\]
Finally, since $\sum_{k\in[d]} |\xi_k|^{2(j+1)}\le (\sum_{k=1}^d \xi_k^2)^{(j+1)}$ we deduce that  the inequality
\[
\big|t^{n+1}\,\partial_t^{n+1}\mathcal F_{\ZZ^d}g_t(\xi)\big|
\lesssim_n \sum_{j=0}^n \big( t^{j+1}|\xi|^{2(j+1)}+|\xi|^2 \big)
\exp\big(-\pi t |\xi|^22^{-n}/ 10 \big)
\lesssim_n \exp\big(-\pi t |\xi|^22^{-n-1}/ 10 \big),
\]
holds uniformly in $d\in\ZZ_{+}$, $t\ge 15$ and $\xi \in \TT^d$ as
desired, and the proof of Proposition \ref{pro:est_der_gauss} follows.
\end{proof}

\section{Fourier transform estimates at infinity: small scales} \label{sec:Festsmall}

We first prove an analogue of Proposition~\ref{pro:est_inf} for small scales $t>0$.

\begin{proposition} \label{pro:1}
For every fixed  $D>0$ there exists a small constant $c_{D} > 0$ such that
\begin{align} \label{id:125}
\mathcal F_{\ZZ^d}g_t(\xi)
\le  
\exp\big(-c_{D} e^{-\pi/t} |\xi|^2  \big),
\end{align}
holds for all $d \in \mathbb{Z}_{+}$, $t\le D$ and $\xi\in\TT^d$.
\end{proposition}

\begin{proof}
By the product structure of $\mathcal F_{\ZZ^d}g_t(\xi)$ it suffices to prove \eqref{id:125} for $d=1$.
Using \eqref{eq:ps1} and \eqref{eq:ps2} we see that 
\begin{align*}
1 - \frac{\theta_t(\xi)}{\theta_t(0)}=
\theta_{1/t}(0)^{-1}\sum_{j \in \ZZ} e^{-\pi j^2/t}
\big(1 - \cos (2\pi j \xi) \big).
\end{align*}
Using the facts that $1 - \cos 2x = 2\sin^2 x \le 2x^2$ for $x \in \RR$, and $\sin^2 x \simeq x^2$ for $x \in [-\pi/2, \pi/2]$, and the estimate
\[
\sum_{j \in \ZZ} e^{-\pi j^2/t} j^2
\simeq_{D} 
e^{-\pi/t}, \qquad  t \le D,
\]
and Lemma~\ref{lem:1}, we obtain
\begin{align*}
1 - \frac{\theta_t(\xi)}{\theta_t(0)}
\simeq_{D} 
e^{-\pi/t} \xi^2, \qquad t \le D, \quad \xi \in \TT.
\end{align*}
This means that there is a small constant $c_{D} > 0$ such that
\begin{align*}
1 - \frac{\theta_t(\xi)}{\theta_t(0)}
\ge
c_{D}
e^{-\pi/t} \xi^2, \qquad t \le D, \quad \xi \in \TT.
\end{align*}
Taking into account the inequality $1 - x \le e^{-x}$ for $x \ge 0$, we obtain 
\begin{align*}
\frac{\theta_t(\xi)}{\theta_t(0)}
\le
1 - c_{D} e^{-\pi/t} \xi^2
\le
\exp\big(-c_{D} e^{-\pi/t} \xi^2  \big), 
\qquad t \le D, \quad \xi \in \TT.
\end{align*} 
This finishes the proof of Proposition~\ref{pro:1}.
\end{proof}
 
Next, we want to show the following important result.

\begin{proposition}
\label{pro:est_dersmall}
For every fixed $D>0$ and for each $n\in \NN$ we have that
\[
\big|t^{2n} \partial_t^n \mathcal F_{\ZZ^d}g_t(\xi)\big| \lesssim_{n,D} 1,
\]
holds uniformly in $d\in\ZZ_{+}$, $t\le D$ and $\xi\in \TT^d$.
\end{proposition}

We begin with a one-dimensional variant of Proposition~\ref{pro:est_dersmall}, which reads as follows.
\begin{lemma}
\label{lem:G12.1}
For each $n\in \ZZ_+$ and $D>0$ fixed we have that
\begin{equation}
\label{G12.1}
\bigg| \partial_t^n \bigg(\frac{\theta_t(\xi)}{\theta_t(0)}\bigg)\bigg|
\lesssim_{n,D} t^{-2n} e^{-\pi/t} \xi^2 \exp\big(- e^{-\pi/t} \xi^2 \big),
\end{equation}
holds 
uniformly in $t\le D$ and $\xi\in \TT$.
\end{lemma}

Notice that inequality \eqref{G12.1} is not true for $n=0$; it suffices to take $\xi =0$.

\begin{proof}[Proof of Lemma~\ref{lem:G12.1}]
Observe first that the factor $\exp(- e^{-\pi/t} \xi^2 )$ in
\eqref{G12.1} is irrelevant since $e^{-\pi/t} \xi^2 \le 1$. However, it is useful to keep this factor in the statement of the lemma for the purpose of proving Proposition \ref{pro:G14.1} below. The proof of  Lemma~\ref{lem:G12.1} is technical, so we will proceed in a few steps.

\paragraph{\bf Step 1}
By \eqref{eq:ps1} and \eqref{eq:ps2}, notice that for any
$n \in \ZZ_+$ we have
\begin{align*}
\partial_t^n\,\bigg(\frac{\theta_t(\xi)}{\theta_t(0)}\bigg)
=
\partial_t^n\,\bigg(\frac{\theta_t(\xi)}{\theta_t(0)} - 1 \bigg)
=
\partial_t^n\,\bigg(\frac{J_t(\xi)}{\theta_{1/t} (0)}\bigg),
\end{align*}
where
\begin{align*}
J_t(\xi)
=
2 \sum_{j\in\ZZ_+}
e^{-\pi j^2/t}
\big(\cos (2\pi j \xi) - 1\big).
\end{align*}
By the Leibniz rule we may write
\begin{align}
\label{id:110}
\partial_t^n\,\bigg(\frac{J_t(\xi)}{\theta_{1/t} (0)}\bigg)
=
\sum_{k=0}^{n}{n \choose k}  \partial_t^{k} \big(J_t(\xi)\big) 
\partial_t^{n-k} \bigg(\frac{1}{\theta_{1/t} (0)}\bigg).
\end{align}
In order to prove Lemma~\ref{lem:G12.1} it suffices to show that for
every $k\in\NN_{\le n}$ we have
\begin{align} \label{G13.1}
\bigg| \partial_t^{n-k} \bigg(\frac{1}{\theta_{1/t} (0)}\bigg) \bigg|
& \lesssim_{n,D}
t^{-2(n-k)} e^{-\pi/t} \ind{n-k\ge 1} + \ind{n-k = 0}, 
\qquad t \le D, \\ \label{G13.2}
\big|  \partial_t^{k} (J_t(\xi))  \big|
& \lesssim_{n,D}
t^{-2k} e^{-\pi/t} \xi^2, \qquad t \le D, \quad \xi \in \TT.
\end{align}
Indeed, combining these two estimates with \eqref{id:110} we see that the left hand side of \eqref{G12.1} is controlled by
\begin{align*} 
\sum_{k=0}^{n} t^{-2k} e^{-\pi/t} \xi^2  
\Big( t^{-2(n-k)} e^{-\pi/t} \ind{n-k\ge 1} + \ind{n-k = 0} \Big)
\simeq_{n,D} 
t^{-2n} e^{-\pi/t} \xi^2,
\end{align*}
uniformly in $t \le D$ and $\xi \in \TT$. This proves \eqref{G12.1} and consequently it suffices to justify \eqref{G13.1} and \eqref{G13.2}. 

\paragraph{\bf Step 2} In order to establish estimate \eqref{G13.1} we first show that for any fixed $n \in \ZZ_+$ we have
\begin{equation}
\label{G11.1}
\big|\partial_t^n \big(\theta_{1/t} (0) \big) \big|
\lesssim_{n,D}
t^{-2n} e^{-\pi/t}, \qquad t \le D.
\end{equation}
Indeed, we have
\begin{align} \label{id:109}
\partial_t^n \big(\theta_{1/t} (0) \big)
=
2 \sum_{j\in\ZZ_+}
\partial_t^n \big( e^{-\pi j^2/t} \big).
\end{align}
By the Fa\'a di Bruno formula for the $n$-th derivative  we obtain
\begin{align*}
\partial_t^n \big( e^{-\pi j^2/t} \big)
=
\sum C_{m,n} e^{-\pi j^2/t} 
\prod_{k\in[n]} \bigg(\frac{j^2}{t^{k+1}}\bigg)^{m_k},
\end{align*}
where the summation above runs over all $n$-tuples  $m=(m_1,\ldots,m_{n})\in\NN^n$ satisfying the constraint
$\sum_{k\in[n]} k m_k=n$ and 
$C_{m,n}$ are some universal constants.
Note that
\begin{align*}
\prod_{k\in[n]} \bigg(\frac{j^2}{t^{k+1}}\bigg)^{m_k}
\lesssim_{n,D}
t^{-2n} j^{2n}, \qquad j \in\ZZ_+, \quad t \le D,
\end{align*}
since $\sum_{k\in[n]} m_k \le \sum_{k\in[n]} k m_k=n$. Consequently, we obtain
\begin{align} \label{G11.2}
	\big| \partial_t^n \big( e^{-\pi j^2/t} \big) \big|
	\lesssim_{n,D}
	t^{-2n} j^{2n} e^{-\pi j^2/t}, \qquad j \in\ZZ_+, \quad t \le D.
\end{align}
This together with \eqref{id:109} and a simple estimate
$x^n e^{-x} \lesssim_{n} e^{-x/2}$ for $x>0$, implies that
\begin{align*} 
\big|\partial_t^n \big(\theta_{1/t} (0) \big) \big|
& \lesssim_{n,D}
t^{-2n} \sum_{j\in\ZZ_+} j^{2n} e^{-\pi j^2/t}\\
&\lesssim_{n}
t^{-2n} e^{-\pi /t}
+
t^{-n} \sum_{j\ge 2} e^{-\pi j^2/(2t)} \\
& \lesssim_{D}
t^{-2n} e^{-\pi /t}
+
t^{-n} e^{-2\pi/t}\\
&
\simeq_{n,D}
t^{-2n} e^{-\pi /t}, 
\end{align*}
holds uniformly in $t \le D$. This proves \eqref{G11.1} as desired.

\paragraph{\bf Step 3} Now we are ready to show \eqref{G13.1}. The
case $k=n$ follows easily from Lemma~\ref{lem:1} so from now on we
assume that $k\in\NN_{<n}$. Applying Fa\'a di Bruno's formula for
the $(n-k)$-th derivative we obtain
\begin{align*}
\partial_t^{n-k} \bigg(\frac{1}{\theta_{1/t} (0)}\bigg)
=
\sum D_{m,n-k} \big(\theta_{1/t} (0)\big)^{-m_1-\ldots-m_{n-k} - 1} 
\prod_{l\in [n-k]} \Big( \partial_t^l \big(\theta_{1/t} (0) \big) \Big)^{m_l},
\end{align*}
where the summation above runs over all $(n-k)$-tuples
$m=(m_1,\ldots,m_{n-k})\in \NN^{n-k}$ satisfying the constraint
$\sum_{l\in [n-k]} l m_l=n-k$ and
$D_{m,n-k}$ are some universal constants. 
Using now Lemma~\ref{lem:1} and \eqref{G11.1} we obtain
\eqref{G13.1}.

\paragraph{\bf Step 4} Here we establish estimate \eqref{G13.2}. 
By $1 - \cos 2x = 2\sin^2 x \le 2x^2$ and \eqref{G11.2}, we obtain that
\begin{align*}
\big|  \partial_t^{k} (J_t(\xi))  \big|
& \lesssim_{n,D}
\sum_{j \in \ZZ_+} 
t^{-2k} j^{2k} e^{-\pi j^2/t} j^2 \xi^2
=
t^{-2k} \xi^2 \Big( e^{-\pi/t} 
+ \sum_{j\ge 2} j^{2k+2} e^{-\pi j^2/t}  \Big)  
\simeq_{n,D}
t^{-2k} e^{-\pi/t} \xi^2, 
\end{align*}
holds uniformly in $t \le D$ and $\xi \in \TT$. This proves \eqref{G13.2} and the proof of Lemma~\ref{lem:G12.1} is complete.
\end{proof}

We are now ready to prove Proposition~\ref{pro:est_dersmall}.
Actually we shall prove a stronger statement.

\begin{proposition}
\label{pro:G14.1}
For any  $D>0$ and for each $n\in \NN$ there exist
two constants $c_{n,D}>0$ and $C_{n,D}>0$ independent of the dimension $d\in\ZZ_+$
 such that
\begin{align}
\label{G14.1}
\big|t^{2n} \partial_t^n \mathcal F_{\ZZ^d}{g_t}(\xi)\big|
\le C_{n,D} \exp\big(-c_{n,D} e^{-\pi/t} |\xi|^2 \big),
\end{align}
holds for all $t\le D$ and $\xi\in \TT^d$.
\end{proposition}

\begin{proof}
We will proceed by induction on $n\in\NN$ in a similar way as in the proof of
Proposition~\ref{pro:est_der_gauss}.   Notice that for
$n=0$ the estimate \eqref{G14.1} follows from
Proposition~\ref{pro:1}. 
Further, using Lemma~\ref{lem:G12.1} we get \eqref{G14.1} for $d=1$ and any $n \in \ZZ_{+}$.
Assume now that \eqref{G14.1} is  true for
$n\in\NN_{\le N}$ for some $N \in \NN$. We show that it  also holds for
$n=N+1$.  Combining \eqref{eq:pro:est_der_gauss_1} with
Lemma~\ref{lem:G12.1} and the induction hypothesis we see that
\begin{equation*}
\begin{split}
t^{2(N+1)} |\partial_t^{N+1}\mathcal F_{\ZZ^d}{g_t}(\xi)| 
& =
\Big|
\sum_{k\in[d]} \sum_{n=0}^{N} {N \choose n} \left[t^{2(n+1)}\partial_t^{n+1} \bigg(\frac{\theta_t(\xi_k)}{\theta_t(0)}\bigg)\right] \left[t^{2(N-n)}\partial_t^{N-n}\mathcal F_{\ZZ^d}{g_t}(\xi^{(k)})\right] \Big| \\
& \lesssim_{N,D} \sum_{k\in[d]} \sum_{n=0}^{N} e^{-\pi/t} \xi_k^2 \exp\big(- e^{-\pi/t} \xi_k^2 \big) \exp\big(-c_{N-n,D} e^{-\pi/t} |\xi^{(k)}|^2 \big),
\end{split}
\end{equation*}
holds uniformly in $d\in\ZZ_{+}$, $t\le D$ and $\xi\in \TT^d$, where
$\xi^{(k)}=(\xi_1,\ldots,\xi_{k-1},\xi_{k+1},\ldots,\xi_d)$ for
$k\in[d]$.  Let
$c_{N+1,D} := \frac{1}{2}\min_{ n \in \NN_{\le N}}{(c_{n,D} \wedge 1)}$ and notice
that we can further bound the above expression by
\begin{equation*}
\begin{split}
\sum_{k\in[d]} e^{-\pi/t} \xi_k^2 \exp\big(-2c_{N+1,D} e^{-\pi/t} |\xi|^2 \big) & =
e^{-\pi/t} |\xi|^2 \exp\big(-2c_{N+1,D} e^{-\pi/t} |\xi|^2 \big) \\
& \lesssim_{N,D} \exp\big(-c_{N+1,D} e^{-\pi/t} |\xi|^2 \big),
\end{split}
\end{equation*}
uniformly in $d\in\ZZ_{+}$, $t\le D$ and $\xi\in \TT^d$. This proves Proposition~\ref{pro:G14.1}.
\end{proof}

As a simple consequence of Proposition~\ref{pro:est_dersmall}, we obtain the following.

\begin{corollary}
\label{cor:est_dersmall}
Let $\psi \colon (0,\infty) \to (0,1)$ be the increasing bijection given by
\begin{equation}
\label{def:psi}
\psi(t):= e^{-\pi/t}, \qquad t>0.
\end{equation}
Then for any  $c \in (0,1)$ and for each $n\in \NN$ we have that
\[
\big|t^{n}  \partial_t^n \mathcal F_{\ZZ^d}{g_{\psi^{-1}(t)}}(\xi)\big|
\lesssim_{n,c} 1,
\]
holds uniformly in $d\in\ZZ_{+}$, $t\le c$ and $\xi\in \TT^d$.	
\end{corollary}

\begin{proof}
Notice that $\psi^{-1}(u) = - \frac{\pi}{\log u}$ for
$u \in (0,1)$. Observe that for $n=0$ this result follows directly
from Proposition~\ref{pro:est_dersmall}.  Let us fix
$n \in \mathbb{Z}_+$ and $c \in (0,1)$. By Fa\'a di Bruno's formula \eqref{eq:3} it suffices to 
verify that for every tuple
$(m_1, \ldots, m_n) \in \mathbb{N}^n$ satisfying
$m_1 +2m_2+\ldots + n m_n = n$ we have
\begin{align*}
\Big| \partial_s^{m_1+\ldots+m_n}  \mathcal F_{\ZZ^d}{g_{s}} (\xi) \big|_{s = \psi^{-1}(t)} \Big|
\prod_{k\in[n]} \big| (\psi^{-1})^{(k)} (t) \big|^{m_k}
\lesssim
t^{-n}, \qquad t \le c, \quad \xi \in \TT^d.
\end{align*} 
This, in turn, is equivalent to showing that
\begin{align}
\label{red1}
\big| \partial_u^{m_1+\ldots+m_n} \mathcal F_{\ZZ^d}{g_{u}} (\xi)
\big| \prod_{k\in[n]} \big| (\psi^{-1})^{(k)} (\psi(u)) \big|^{m_k}
\lesssim (\psi(u))^{-n}, \qquad u \le \psi^{-1}(c), \quad \xi \in \TT^d.
\end{align} 
Observe that $|\psi' (u)| \simeq u^{-2} e^{-\pi/u}$ for $u>0$, and for any $j \in \ZZ_+$ we have 
\[
|\psi^{(j)} (u)| \lesssim_{j, c} u^{-2j} e^{-\pi/u}, \qquad u \le \psi^{-1}(c).
\]

Consequently, combining this with \eqref{id:104} for every $k \in \ZZ_+$ (with $T(k)$ as in \eqref{id:104}) we obtain
\begin{align*}
\big| (\psi^{-1})^{(k)} (\psi(u)) \big|
\lesssim_{k}
u^{2(2k-1)} e^{\pi(2k-1)/u} \sum_{s \in T(k)} 
\Big( \prod_{j\in[k]} u^{-2js_j} e^{-\pi s_j/u} \Big)
\lesssim_{k}
u^{2} e^{\pi k/u}, \qquad u \le \psi^{-1}(c).
\end{align*}
By Proposition~\ref{pro:est_dersmall} (with $D = \psi^{-1}(c)$) we see
that the left-hand side of \eqref{red1} is controlled by
\begin{align*} 
\big| \partial_u^{m_1+\ldots+m_n} \mathcal F_{\ZZ^d}{g_{u}}(\xi)  \big|
\prod_{k \in [n]} \big( u^{2} e^{\pi k/u} \big)^{m_k}
=
u^{2(m_1+\ldots+m_n)} \big| \partial_u^{m_1+\ldots+m_n}  \mathcal F_{\ZZ^d}{g_{u}}(\xi)  \big|
e^{\pi n/u}
\lesssim_n
(\psi(u))^{-n}, 
\end{align*} 
uniformly in $u \le \psi^{-1}(c)$ and $\xi \in \TT^d$. This proves \eqref{red1} and Corollary~\ref{cor:est_dersmall} is justified.
\end{proof}

\section{Dimension-free estimate for the difference operator} \label{sec:implp}
Let $\{m_t\}_{t > 0}$ be a family of measurable, bounded complex-valued functions on $\TT^d$ such that for each $\xi \in \TT^d$ the mapping
$(0, \infty) \ni t \mapsto m_t(\xi)$ is smooth. 
By the Carath\'eodory measurability theorem, see \cite[Theorem~20.15]{AB_book}, the function $m_t(\xi)$ and all its $t$-derivatives are jointly measurable.
Consider the family of
multiplier operators $\{M_t\}_{t > 0}$ defined by
\begin{align}
\label{eq:5}
\mathcal F_{\ZZ^d}{M_t f}(\xi):=m_{t}(\xi)\mathcal F_{\ZZ^d}{f}(\xi),\qquad \xi \in \TT^d.
\end{align}
For $n\in \NN$ we let 
\begin{equation}
\label{eq:Bn}
 B_n(\{m_s\}_{s > 0})
 :=
 \sup_{j\in\NN_{\le n}}\sup_{s> 0}\sup_{\xi\in \TT^d}
\left|s^j \partial_s^j m_s(\xi)\right|+\sup_{s >0}\|M_s\|_{\ell^1\to \ell^1}
+1.
\end{equation}
We added a constant $1$ in \eqref{eq:Bn} only to ensure that $B_n(\{m_s\}_{s >0}) \ge 1$. By the definition
\[
B_0(\{m_s\}_{s > 0})\le B_1(\{m_s\}_{s > 0})\le \ldots\le B_n(\{m_s\}_{s > 0}).
\]

\begin{condition}
	Let $n \in \NN$ be fixed. 
We say that the family $\{m_t\}_{t>0}$ satisfies condition $(B_n)$, if
\[
B_n(\{m_s\}_{s>0})<\infty.
\]
\end{condition}

The following theorem, which is the main result of this section, will
be crucial to prove Theorem \ref{thm:main} for all
$p\in(1, \infty)$. In what follows $q>1$ is the dual exponent to
$p>1$, i.e.\ $1/p+1/q=1.$
\begin{theorem}
\label{thm:implp}
Let $\{M_t\}_{t > 0}$ be a family of operators as in \eqref{eq:5}
with $\{m_t\}_{t>0}$
satisfying condition $(B_n)$ for some $n\in \ZZ_+$. Take $1<p\le 2$
such that $2n/q\le 1$. Then for any $\alpha\in(0,2n/q)$ we have
\begin{equation}
\label{eq:thm:implp}
\left\|M_{t+w}-M_t\right\|_{\ell^p\to \ell^p}
\leq 
C(n,p,\alpha) B_n(\{m_s\}_{s > 0})
\left({w}{t}^{-1}\right)^{\alpha}, \qquad w, t > 0,
\end{equation}
with the constant $C(n,p,\alpha)$ depending only on $n,p$ and
$\alpha$, but independent of the dimension $d\in\ZZ_+$.
\end{theorem}
\begin{remark}
\label{rem:thm:implp}
Now a few remarks are in order.
\begin{enumerate}[label*={\arabic*}.]
\item An analogous result holds for multiplier operators on
$\RR^d$. The proof is essentially the same. It is important to observe
that when dealing with maximal functions for averages over symmetric
convex bodies in $\mathbb{R}^d$, the interesting case arises when the
multipliers being considered are dilations of a single
multiplier. This is not the case on the torus $\TT^d$.

\item For $p=2$, an application of Plancherel's theorem  gives the following improvement of \eqref{eq:thm:implp}, i.e.
\begin{equation}
\label{eq:thm:implp:l2}
\left\|M_{t+w}-M_t\right\|_{\ell^2\to \ell^2}\leq B_1(\{m_s\}_{s > 0}) {w}{t}^{-1},\qquad w, t > 0.
\end{equation}

\item Since $M_t$ is an $\ell^1(\ZZ^d)$ bounded convolution operator we have
$\|M_t\|_{\ell^p\to\ell^p}\le \|M_t\|_{\ell^1\to\ell^1}$ yielding
\[
\left\|M_{t+w}-M_t\right\|_{\ell^p\to \ell^p} \le 2 B_0 (\{m_s\}_{s>0}).
\]
 Hence, \eqref{eq:thm:implp} is obvious for $w>t/2$, and  in the proof of Theorem~\ref{thm:implp} we can assume that $w\le t/2.$
                
\end{enumerate}

\end{remark}

The remainder of this section is devoted to proving
Theorem~\ref{thm:implp}. Inequality \eqref{eq:thm:implp} will be
handled by using fractional integration, which we will recall with
simple facts from \cite{DGM1}.

Fix $\xi\in\TT^d$ and let $h(v) = m_v(\xi)/v$ for
$v\in (0,\infty)$. Then by the Leibniz formula and \eqref{eq:Bn} we have 
\begin{equation}
\label{eq:lem:Pua:1}
\left|\partial_s^n \left(\frac{m_s(\xi)}{s}\right)\right| 
\lesssim_n 
B_n(\{m_s\}_{s >0})\, s^{-n-1}, \qquad s>0, \quad \xi \in \TT^d.
\end{equation}
Notice that $h$ is
locally Lipschitz on $(0,\infty)$, i.e.\ it is Lipschitz on $[t_{0},\infty)$ for every $t_{0} > 0$. 
Thus by \cite[Lemma 6.11]{DGM1} applied with any $t_0 > 0$, for any $\alpha\in(0,1)$ we have 
\[
h(v) = \frac{1}{\Gamma(\alpha)}\int_v^{\infty}(u-v)^{\alpha-1}D^{\alpha} h(u)\,du, \qquad v > 0,
\]
where $D^{\alpha}h$ is the fractional derivative given by
\begin{align}
\label{id:113}
D^{\alpha}h(u) = -\frac{1}{\Gamma(1-\alpha)}\int_u^{\infty} (s-u)^{-\alpha}h'(s)\,ds.
\end{align}
Note that in the statement of \cite[Lemma 6.11]{DGM1} there is a misprint; it should be $h(t)$ in place of $t$ on the right-hand side there.
Recalling the definition of $h$ we see that
\[
m_t(\xi)=\frac{1}{\Gamma(\alpha)}\int_t^{\infty} t (u-t)^{\alpha-1}D^{\alpha}_v\left[\frac{m_v(\xi)}{v}\right]\,\bigg|_{v=u}\, du, \qquad t>0,
\]
where $D^{\alpha}_v[\ \cdot\ ]|_{v=u}$ denotes the
fractional derivative in variable $v$ evaluated at $u$.  We can further write
\begin{align*}
m_t(\xi)
=
\frac{1}{\Gamma(\alpha)}\int_t^{\infty} \frac{t}{u}\left(1-\frac{t}{u}\right)^{\alpha-1}
\left[u^{\alpha+1}D^{\alpha}_v\left[\frac{m_v(\xi)}{v}\right]\,\bigg|_{v=u}\right] \, \frac{du}{u}
=
\int_0^{\infty}A(t,u)p_u^{\alpha}(\xi)\,du,
\end{align*}
with
\begin{align*}
A(t,u):=
\ind{u > t > 0}
\frac1{\Gamma(\alpha)}\frac{t}{u} \Big(1- \frac{t}{u} \Big)^{\alpha-1} \frac1u, 
\qquad t,u>0.
\end{align*}
and the multiplier 
\[
p_u^{\alpha}(\xi)
:=
u^{\alpha+1}D^{\alpha}_v\left[\frac{m_v(\xi)}{v}\right]\,\bigg|_{v=u},
\qquad \xi \in \TT^d.
\]
Observe that thanks to \eqref{id:113} we have
\begin{equation}
\label{eq:thm:implp:1}
	D^{\alpha}_v\left[\frac{m_v(\xi)}{v}\right]\,\bigg|_{v=u}=-\frac{1}{\Gamma(1-\alpha)}\int_u^{\infty}(s-u)^{-\alpha}\partial_s \left(\frac{m_s(\xi)}{s}\right)\, ds.
\end{equation}
By inequality \eqref{eq:lem:Pua:1} and condition $(B_n)$ we obtain that 
\begin{align*}
|p_u^{\alpha}(\xi)|
& \lesssim_{\alpha} 
B_1(\{m_s\}_{s>0})  u^{\alpha+1}\int_u^{\infty} (s-u)^{-\alpha}s^{-2}\,ds  \\
& =  
B_1(\{m_s\}_{s>0})  \int_1^{\infty} (s-1)^{-\alpha}s^{-2}\,ds
\simeq_{\alpha}  
B_1(\{m_s\}_{s>0}),
\qquad u>0, \quad \xi \in \TT^d,
\end{align*}
for $\alpha\in(0,1)$. In other words, the multiplier $p_u^{\alpha}$ is well defined on $\TT^d$ and bounded uniformly in $u>0$. 
Notice that the measurability of  $p_u^{\alpha}$ follows from Fubini's theorem. 
Thus 
for any fixed $\alpha\in(0,1)$ the family of multiplier operators
\[
\mathcal F_{\ZZ^d} P^{\alpha}_u f(\xi)
:=
p_u^{\alpha}(\xi)\mathcal F_{\ZZ^d}{f}(\xi),
\qquad \xi \in \TT^d,
\]
is bounded on $\ell^2(\ZZ^d)$ uniformly in $u>0$. In particular, we have
\begin{align*}
\int_0^{\infty}A(t,u)\|P_u^{\alpha}\|_{\ell^2\to \ell^2}\,du 
& \lesssim_{\alpha}  
B_1(\{m_s\}_{s>0})\int_t^{\infty}A(t,u)\,du \\
&\lesssim_{\alpha} 
B_1(\{m_s\}_{s>0})\int_1^{\infty}u^{-1-\alpha}(u-1)^{\alpha-1} \, du \\
&\simeq_{\alpha}  
B_1(\{m_s\}_{s>0}),
\end{align*}
and thus, for each $\alpha\in(0,1)$ and $t > 0$ the formula
\begin{align}
\label{eq:7}
M_t f=\int_0^{\infty} A(t,u) P_u^{\alpha} f \,du 
\end{align}
makes sense at least for $f\in \ell^2(\ZZ^d)$ as a Bochner integral.

\begin{lemma}
\label{lem:dif}
Let $\{M_t\}_{t > 0}$ be a family of operators as in
\eqref{eq:5}. Then for any $p\in[1, \infty]$ and $\alpha\in(0, 1)$ and every $f\in\ell^p(\ZZ^d)\cap\ell^2(\ZZ^d)$ we
have that
\begin{align}
\label{eq:6}
\left\|M_{t+w}f-M_t f\right\|_{\ell^p}\lesssim_{\alpha}
\left(w{t}^{-1}\right)^{\alpha}
\sup_{u > 0}\big\|P_u^{\alpha} f\big\|_{\ell^p},
\end{align}
whenever $t>0$ and $0<w\le t/2$.
\end{lemma}
\begin{proof}
Using representation from \eqref{eq:7} and  Minkowski's integral inequality we obtain
\begin{align*}
\left\|M_{t+w}f-M_t f\right\|_{\ell^p}
\le\sup_{u\ge t}\big\|P_u^{\alpha} f\big\|_{\ell^p}  \int_0^{\infty} \Delta_{t,w}(u) \,du,
\end{align*}
where $\Delta_{t,w}(u):=|A(t+w,u)-A(t,u)|$ for $t>0$ and $0< w \le t/2$. For any $\alpha \in (0,1)$ we show that
\begin{align}
\label{eq:8}
\int_0^{\infty} \Delta_{t,w}(u) \,du
\lesssim_{\alpha} 
\left(w{t}^{-1}\right)^{\alpha}.
\end{align}
This will imply \eqref{eq:6} as desired. To obtain the last bound, note that
\begin{align} \label{id:103}
\Delta_{t,w}(u) \lesssim_{\alpha}
\begin{cases}
t |u-t|^{\alpha-1}u^{-\alpha-1},& \text{ if } t\le u\le t+w,\\
t|u-t-w|^{\alpha-1}u^{-\alpha-1},& \text{ if } t+w\le u\le t+2w,\\
w |u-t|^{\alpha-2}u^{-\alpha},& \text{ if } t+2w\le u.
\end{cases}
\end{align}
Indeed, the estimate in the first region is trivial. For $t+w\le u\le t+2w$ we have
\begin{equation*}	
\Delta_{t,w}(u)
\lesssim_{\alpha} 
t |u-t|^{\alpha-1}u^{-\alpha-1}+(t+w)|u-t-w|^{\alpha-1}u^{-\alpha-1}
\simeq_{\alpha} 
t|u-t-w|^{\alpha-1}u^{-\alpha-1}.
\end{equation*}
Finally, for $t+2w\le u$ we use the mean value theorem with $\zeta \in [t/u,(t+w)/u]$ and we obtain
\begin{equation*}
\begin{split}
\Delta_{t,w}(u) & = \left|\frac tu\left(1-\frac tu\right)^{\alpha-1}-\frac{t+w}{u}\left(1-\frac{t+w}{u}\right)^{\alpha-1}\right|
\frac{1}{u\,\Gamma(\alpha)} \\
& \simeq_{\alpha} wu^{-2} \big(s(1-s)^{\alpha-1} \big)'(\zeta) \simeq_{\alpha} w u^{-2} (1-\zeta)^{\alpha-2} \leq
w u^{-2} \left(1-\frac{t+w}{u}\right)^{\alpha-2} \\
& = w|u-t-w|^{\alpha-2}u^{-\alpha} \simeq_{\alpha} w |u-t|^{\alpha-2}u^{-\alpha}.
\end{split}
\end{equation*}
This justifies \eqref{id:103}. Then by \eqref{id:103} the integral from \eqref{eq:8} can be dominated by 
\begin{align*}
\int_t^{t+w} t|u-t|^{\alpha-1}u^{-\alpha-1}\,du+\int_{t+w}^{t+2w} t|u-t-w|^{\alpha-1}u^{-\alpha-1}\,du+w\int_{t+2w}^{\infty}|u-t|^{\alpha-2}u^{-\alpha}\,du,
\end{align*}
which in turn can be dominated by $t^{-\alpha}w^{\alpha} + wt^{-1}\simeq_{\alpha}\left(w{t}^{-1}\right)^{\alpha}$ as desired. 
\end{proof}

By Lemma \ref{lem:dif} the proof of Theorem \ref{thm:implp} is reduced to controlling $\sup_{u > 0}\|P_u^{\alpha} \|_{\ell^p\to \ell^p}.$
\begin{lemma}
\label{lem:Pua}
Let $\{M_t\}_{t > 0}$ be a family of operators as in \eqref{eq:5}
with $\{m_t\}_{t>0}$
satisfying condition $(B_n)$ for some $n\in \ZZ_+$. Take $1<p\le 2$
such that $2n/q\le 1$. Then for any $\alpha\in(0,2n/q)$ we have
\begin{equation}
\label{eq:lem:Pua}
\sup_{u > 0}\|P_u^{\alpha} \|_{\ell^p\to \ell^p}
\leq 
C(n,p,\alpha) B_n(\{m_s\}_{s > 0}),
\end{equation}
with the constant $C(n,p,\alpha)$ depending only on $n,p$ and
$\alpha$, but independent of the dimension $d\in\ZZ_+$.
\end{lemma}
Lemma \ref{lem:dif} and Lemma \ref{lem:Pua}  imply Theorem \ref{thm:implp}. Thus, from now on we focus on justifying Lemma \ref{lem:Pua}.

\begin{proof}[Proof of Lemma \ref{lem:Pua}]
Assume that condition $(B_n)$ holds for some $n\in\ZZ_+$. We will
proceed in a few steps for the sake of clarity.  By
\eqref{eq:thm:implp:1}, the multiplier $p_u^{\alpha}(\xi)$
corresponding to $P_u^{\alpha}$ is equal to
\[
p_u^{\alpha}(\xi)= -\frac{u^{\alpha+1}}{\Gamma(1-\alpha)}\int_u^{\infty}(s-u)^{-\alpha}\partial_s \left(\frac{m_s(\xi)}{s}\right)\, ds.
\]
\paragraph{\bf Step 1} 
The family $\{m_t\}_{t>0}$ satisfies condition $(B_n)$, then 
integration by parts $(n-1)$-times gives
\begin{align*}
p_u^{\alpha}(\xi)
& =
-\frac{u^{\alpha+1}}{\Gamma(1-\alpha)}\int_u^{\infty}\partial_s\left(\frac{(s-u)^{-\alpha+1}}{1-\alpha}\right)\partial_s \left(\frac{m_s(\xi)}{s}\right)\, ds\\
&=
\frac{u^{\alpha+1}}{(1-\alpha)\Gamma(1-\alpha)}\int_u^{\infty}(s-u)^{-\alpha+1}\partial_s^2 \left(\frac{m_s(\xi)}{s}\right)\, ds \\
&=
\frac{u^{\alpha+1}}{\Gamma(2-\alpha)}\int_u^{\infty}(s-u)^{-\alpha+1}\partial_s^2 \left(\frac{m_s(\xi)}{s}\right)\, ds\\
&=\ldots
=
\frac{(-1)^nu^{\alpha+1}}{\Gamma(n-\alpha)}\int_u^{\infty}(s-u)^{-\alpha+n-1}\partial_s^n \left(\frac{m_s(\xi)}{s}\right)\, ds.
\end{align*}
The above formula allows us to define  a holomorphic extension of $p_u^{\alpha}$ given by
\begin{equation}
\label{eq:lem:Pua:0}
p_u^z(\xi)
:= 
\frac{(-1)^nu^{z+1}}{\Gamma(n-z)}\int_u^{\infty}(s-u)^{-z+n-1}\partial_s^n \left(\frac{m_s(\xi)}{s}\right)\, ds,\qquad \xi\in \TT^d, \quad z\in S,
\end{equation}
on the strip  $S:=\{z\in\CC: -1 < {\rm Re}\,{z} <n\}$. The fact that $S\ni z\mapsto p_u^{z}(\xi)$ is holomorphic  follows from Morera's theorem and Fubini's theorem, since we have \eqref{eq:lem:Pua:1}.

\paragraph{\bf Step 2} Fix $\varepsilon\in (0,1)$ and let
$S_{\varepsilon}:=\{z\in \CC\ \colon -\varepsilon\le {\rm Re}\,{z} \le n-\varepsilon\}.$
The proof of Lemma \ref{lem:Pua} will be based on Stein's complex
interpolation theorem (see e.g.\ \cite[Theorem 1.3.7 and Excercise
1.3.4]{Gra1}) applied to the family of operators
\[
\mathcal F_{\ZZ^d} P_u^{z} f
:=
p_u^{z}  \mathcal F_{\ZZ^d}{f},\qquad z\in S_{\varepsilon}.
\]
Recall that in Stein's complex
interpolation theorem any growth of norms smaller than a
constant times the double exponential
$\exp(\exp(\gamma |{\rm Im}\, z|/n))$ with $0\le \gamma <\pi$ is
admissible. Observe that for $f,g\in \ell^2(\ZZ^d)$ the function
$\langle P_u^z f, g\rangle_{\ell^2}$ is holomorphic in the interior of
$S_\varepsilon$. Indeed, by the Plancherel theorem and Fubini's
theorem we have
\[
\langle P_u^z f, g\rangle_{\ell^2}
=
\frac{(-1)^nu^{z+1}}{\Gamma(n-z)}\int_u^{\infty}(s-u)^{-z+n-1}\int_{\TT^d}\partial_s^n \left(\frac{m_s(\xi)}{s}\right)\,\mathcal F_{\ZZ^d}{f}(\xi)\overline{\mathcal F_{\ZZ^d}{g}(\xi)}\,d\xi\, ds.
\]
The application of Fubini's theorem is allowed by \eqref{eq:lem:Pua:1} since $\mathcal F_{\ZZ^d}{f}\overline{\mathcal F_{\ZZ^d}{g}}\in L^1(\TT^d)$ for $f,g \in \ell^2(\ZZ^d)$. 
Using the bound
\begin{align} \label{id:111} \int_u^{\infty}(s-u)^{-{\rm Re}{z}+n-1} s^{-n-1} \, ds \simeq_{n,\varepsilon} u^{-{\rm Re}{z}-1}, \qquad z\in S_{\varepsilon}, \quad u > 0,
\end{align}
we see that for every $f,g \in \ell^2(\ZZ^d)$ and $z\in S_{\varepsilon}$ we have
\[
\left|\langle P_u^z f, g\rangle_{\ell^2}\right| \lesssim_{n,\varepsilon} \frac{1}{|\Gamma(n-z)|} B_n(\{m_s\}_{s>0}) \|f\|_{\ell^2}\|g\|_{\ell^2} \lesssim_{n,\varepsilon} e^{3\pi |{\rm Im}{z}|/4} C(f,g).
\]
In the last line we have used that ${\rm Re}{(n-z)}\ge \varepsilon$ together with the bound  
\begin{equation}
\label{eq:lem:Pua:2}
|\Gamma(n-z)|
\simeq_{n, \varepsilon}
(|{\rm Im}{z}| + 1)^{n - {\rm Re}{z} - 1/2} e^{- \pi |{\rm Im}{z}|/2}, 
\qquad z\in S_{\varepsilon}, 
\end{equation}
which follows from a well-known asymptotic formula (see e.g.\ \cite[eq.\ 5.11.9]{NIST}) asserting 
\begin{align*}
|\Gamma(x+iy)| 
\simeq_{M,\varepsilon}
(|y| + 1)^{x-1/2} e^{-\pi |y|/2}, \qquad y \in \RR, \quad \varepsilon < x \le M.
\end{align*}
Concluding, we have proved that  for all $f,g\in \ell^2(\ZZ^d)$ the mapping $z\mapsto \langle P_u^z f, g\rangle_{\ell^2}$ defines a holomorphic function of admissible growth in the strip $S_{\varepsilon}.$

\paragraph{\bf Step 3} We shall prove that
\begin{align}
\label{eq:lem:Pua:3}
\|P_u^z\|_{\ell^2\to \ell^2}
&\lesssim_{n,\varepsilon}  
B_n(\{m_s\}_{s>0})   e^{3\pi |{\rm Im}{z}|/4},\qquad {\rm Re}{z}=n-\varepsilon,\\
\label{eq:lem:Pua:4}
\|P_u^z\|_{\ell^1\to \ell^1}
&\lesssim_{\varepsilon}  
B_0(\{m_s\}_{s>0})  e^{3\pi |{\rm Im}{z}|/4},\qquad {\rm Re}{z}=-\varepsilon,	
\end{align}
uniformly in $u>0$. Then, applying Stein's complex interpolation theorem we obtain for each $\theta\in(0,1)$ such that  $1/p=\theta/2+(1-\theta)/1=1-\theta/2,$ the uniform in $u>0$ estimate 
\[
\left\|P_u^{\theta(n-\varepsilon)+(1-\theta)(-\varepsilon)}\right\|_{\ell^p\to \ell^p}
\lesssim_{n,\theta,\varepsilon}B_n(\{m_s\}_{s>0}).
\] 
Since $\theta(n-\varepsilon)+(1-\theta)(-\varepsilon)=\theta n-\varepsilon$ and $\theta=2(p-1)/p=2/q$ the above inequality becomes
\[
\|P_u^{2n/q-\varepsilon}\|_{\ell^p\to \ell^p}
\lesssim_{n,p,\varepsilon}
B_n(\{m_s\}_{s>0}),\qquad u>0.
\]
Finally, since $\varepsilon$ varies over the interval $(0,2n/q) \subseteq (0,1)$, then for each $\alpha\in (0,2n/q)$ we can find $\varepsilon\in(0,2n/q)$ such that $\alpha=2n/q-\varepsilon$ and we have
\[
\|P_u^{\alpha}\|_{\ell^p\to \ell^p}
\lesssim_{n,p,\alpha} 
B_n(\{m_s\}_{s>0}),\qquad u>0,
\]
which yields \eqref{eq:lem:Pua} as desired. We now focus on justifying
\eqref{eq:lem:Pua:3} and \eqref{eq:lem:Pua:4}.

\paragraph{\bf Step 4} We begin with \eqref{eq:lem:Pua:3}. Using
\eqref{eq:lem:Pua:0}, \eqref{eq:lem:Pua:1}, \eqref{id:111} and
\eqref{eq:lem:Pua:2} we obtain, for ${\rm Re}{z}=n-\varepsilon,$ that
\begin{align*}
|p_u^z (\xi)|
&\le  \frac{u^{{\rm Re}{z}+1}}{|\Gamma(n-z)|} \int_u^{\infty}(s-u)^{-{\rm Re}{z}+n-1}\left|\partial_s^n \left(\frac{m_s(\xi)}{s}\right)\right|\, ds\\
&\lesssim_{n} 
B_n(\{m_s\}_{s>0})\frac{u^{{\rm Re}{z}+1}}{|\Gamma(n-z)|}
\int_u^{\infty}(s-u)^{-{\rm Re}{z}+n-1}s^{-n-1}\,ds\\
&\simeq_{n,\varepsilon} 
B_n(\{m_s\}_{s>0})  \frac{1}{|\Gamma(n-z)|} \\
&\lesssim_{n,\varepsilon}   
B_n(\{m_s\}_{s>0})  e^{3\pi |{\rm Im}{z}|/4},
\end{align*}
uniformly in $u > 0$ and $\xi\in \TT^d.$ This estimate together with
Plancherel's theorem gives \eqref{eq:lem:Pua:3}.

\paragraph{\bf Step 5} It remains to prove \eqref{eq:lem:Pua:4}. Since
${\rm Re}{z}=-\varepsilon<0$, then integration by parts $n$-times gives
\[
p_u^{z}(\xi)= -\frac{u^{z+1}}{\Gamma(1-z)}\int_u^{\infty}(s-u)^{-z}\partial_s \left(\frac{m_s(\xi)}{s}\right)\, ds = -\frac{u^{z+1}z}{\Gamma(1-z)}\int_u^{\infty}(s-u)^{-z-1}\frac{m_s(\xi)}{s}\, ds.
\]
Moreover, using \eqref{eq:lem:Pua:1} and \eqref{eq:lem:Pua:2} it is
easy to see that $p_u^z$ is a bounded function of $\xi\in\TT^d.$
Therefore, for ${\rm Re}{z}=-\varepsilon$ we have
\[
P_u^z f= - \frac{z u^{z+1}}{\Gamma(1-z)}\int_u^{\infty} (s-u)^{-z-1} M_s f\,\frac {ds}{s},
\qquad u > 0,
\]
where the right-hand side is a convergent Bochner integral in
$\ell^2(\ZZ^d).$ Thus, for $f\in \ell^2(\ZZ^d)\cap \ell^1(\ZZ^d)$
using again \eqref{eq:lem:Pua:1} and \eqref{eq:lem:Pua:2} together
with the condition $(B_n)$ we obtain, for $u>0$, that
\begin{align*}
\|P_u^z f\|_{\ell^1}
&\le  
B_0(\{m_s\}_{s>0}) \frac{|z| u^{1-\varepsilon}}{|\Gamma(1-z)|}
\int_u^{\infty}(s-u)^{\varepsilon-1}\,\frac{ds}{s}\|f\|_{\ell^1}\\
&\lesssim_{\varepsilon}  
B_0(\{m_s\}_{s>0})  e^{3\pi |{\rm Im}{z}|/4}\int_{1}^\infty(s-1)^{\varepsilon-1}\,\frac{ds}{s}\|f\|_{\ell^1} \\
&\simeq_{\varepsilon}  
B_0(\{m_s\}_{s>0}) e^{3\pi |{\rm Im}{z}|/4} \|f\|_{\ell^1}.
\end{align*} 
This yields \eqref{eq:lem:Pua:4} and completes the proof of Lemma \ref{lem:Pua}.
\end{proof}
	
We close this section by proving the following useful Corollary that
follows from Theorem \ref{thm:implp}.

\begin{corollary}
\label{cor:implp}
Let $\{M_t\}_{t > 0}$ be a family of operators as in \eqref{eq:5}
with $\{m_t\}_{t>0}$
satisfying condition $(B_n)$ for all $n\in \ZZ_+$.  Take $p\in(1,2),$
$\alpha\in(0,1)$ and $n>\alpha q/2.$ Then
\begin{equation}
\label{eq:cor:implp}
\left\|M_{t+w}-M_t\right\|_{\ell^p\to \ell^p}
\leq 
C(n,p,\alpha) B_n(\{m_s\}_{s>0}) 
\left(w{t}^{-1}\right)^{\alpha},\qquad w,t>0,
\end{equation}
with the constant $C(n,p,\alpha)$ depending only on $n,p$ and
$\alpha$, but independent of the dimension $d\in\ZZ_+$.	
\end{corollary}

\begin{proof}
Fix $p,\alpha$ and $n$ as in the statement of the corollary; in
particular $\alpha<2n/q.$
If $2n/q\le 1,$ then \eqref{eq:cor:implp} is a direct
consequence of Theorem \ref{thm:implp}.
If $2n/q >1$, we take $q_n := 2n > q$. Then, $\alpha < 1 = 2n/q_n$
and the
dual exponent to $q_n$, call it $p_n,$ satisfies
$p_n<p.$ Hence, applying Theorem \ref{thm:implp} we obtain
\[
\left\|M_{t+w}-M_t\right\|_{\ell^{p_n}\to \ell^{p_n} } \leq C(n,p_n,\alpha) B_n(\{m_s\}_{s>0}) \left(w{t}^{-1}\right)^{\alpha},\qquad w,t>0.
\]
Interpolating the latter bound with \eqref{eq:thm:implp:l2} we obtain
\eqref{eq:cor:implp}. 
This completes the proof.
\end{proof}

\section{Proof of Theorem~\ref{thm:main}}\label{sec:pfmain}
Our objective is to prove Theorem~\ref{thm:main}. By 
\eqref{eq:688} and \eqref{estt1} it suffices to show the jump \eqref{jumpest} and oscillation \eqref{oscest} estimates. 
Let us first make some reduction related to jump estimate  \eqref{jumpest}. Observe that
\begin{align*}
N_{\lambda}(G_t f : t > 0)^{1/2}
\le
N_{\lambda/2}(G_t f : t \ge 16)^{1/2} + N_{\lambda/2}(G_t f : 0<t \le 16)^{1/2}.
\end{align*}

Let $G_t^{>}:=G_t$ for $t\ge16$, and
$G_t^{<}:=G_{\psi^{-1}(\psi(16)t)}$ for $0<t\le 1$ with the
 bijection $\psi$ from \eqref{def:psi}.  Then observe  that
\[
\lambda N_{\lambda}(G_t f : 0<t \le 16)^{1/2}
 =
\lambda N_{\lambda}(G_t^{<} f : 0<t \le 1)^{1/2}.
\]
The latter rescaling has only been introduced for a technical reason
to enable handling the large and small scale cases in a unified
way. Using \cite[Lemma 1.3]{JSW} we can further write 
\begin{align}
	\begin{split}
		\label{eq:77}
\lambda N_{\lambda}(G_t^{>} f : t \ge 16)^{1/2}
& \lesssim
\lambda N_{\lambda/3}(G_{2^n}^{>} f : n \in \mathbb{I}_{>})^{1/2}
+ 
\Big( \sum_{n \in \mathbb I_{>}} V^2 \big(G_t^{>} f : t\in[2^n, 2^{n+1}] \big)^2\Big)^{1/2}, \\
\lambda N_{\lambda}(G_t f : 0<t \le 16)^{1/2}
& =
\lambda N_{\lambda}(G_t^{<} f : 0< t \le 1)^{1/2} \\
& \lesssim 
\lambda N_{\lambda/3}(G_{2^{n+1}}^{<} f : n\in\mathbb I_{<})^{1/2} 
+ 
\Big( \sum_{n\in\mathbb I_{<}} V^2 \big(G_t^{<} f : t\in[2^n, 2^{n+1}] \big)^2\Big)^{1/2},
\end{split}
\end{align}
where the implicit constants are universal and do not depend on the dimension $d\in\ZZ_+$ and $\mathbb I_{>}:=\ZZ\cap[4, \infty)$ and $\mathbb I_{<}:=\ZZ\cap(-\infty, -1]$.

Proceeding in a similar way as above, and using \cite[(2) from Remark 2.4 and (2.16)]{MSW}, we obtain
\begin{align}
	\begin{split}
	\label{eq:70}
	&	\sup_{J\in\ZZ_+}\sup_{I\in\mathfrak{S}_J( (0,\infty) )}
		\big\| O_{I,J}^2(G_t f : t > 0) \big\|_{\ell^p} \\
	& \qquad	\lesssim 
		\sup_{J\in\ZZ_+}\sup_{I\in\mathfrak{S}_J(\mathbb{I}_{>})}
		\big\| O_{I,J}^2(G_{2^n}^{>} f : n \in \mathbb I_{>}) \big\|_{\ell^p}
		+
		\sup_{J\in\ZZ_+}\sup_{I\in\mathfrak{S}_J(\mathbb{I}_{<})}
		\big\| O_{I,J}^2(G_{2^{n+1}}^{<} f : n \in \mathbb I_{<}) \big\|_{\ell^p}\\
& \qquad	\quad	+
	\Big\|\Big( \sum_{n \in \mathbb I_{>}} V^2 \big(G_{t}^{>} f : t\in[2^n, 2^{n+1}] \big)^2\Big)^{1/2} \Big\|_{\ell^p}
		+
		\Big\|\Big( \sum_{n \in \mathbb I_{<}} V^2 \big(G_{t}^{<} f : t\in[2^n, 2^{n+1}] \big)^2\Big)^{1/2} \Big\|_{\ell^p}.
	\end{split}
\end{align}

The above estimates \eqref{eq:77} and \eqref{eq:70} reduce the proof of Theorem \ref{thm:main} to the following two propositions.

\begin{proposition} \label{pro:dyad}
For each $p\in(1,\infty)$ there exists $C_{p}>0$ such that for every $f\in \ell^p(\ZZ^d)$ we have 
\begin{align}
\label{id:122}
\begin{split}
\sup_{\lambda>0} \big\| \lambda N_{\lambda}(G_{2^n}^{>} f : n \in \mathbb I_{>})^{1/2} \big\|_{\ell^p}
&\le 
C_{p} \|f\|_{\ell^p}, \\
\sup_{\lambda>0} \big\|\lambda N_{\lambda}(G_{2^{n+1}}^{<} f : n \in \mathbb I_{<})^{1/2}\big\|_{\ell^p}
& \le 
C_{p} \|f\|_{\ell^p}, \\
\sup_{J\in\ZZ_+}\sup_{I\in\mathfrak{S}_J(\mathbb{I}_{>})}
\big\| O_{I,J}^2(G_{2^n}^{>} f : n \in \mathbb I_{>}) \big\|_{\ell^p} 
& \le 
C_{p} \|f\|_{\ell^p}, \\
\sup_{J\in\ZZ_+}\sup_{I\in\mathfrak{S}_J(\mathbb{I}_{<})}
\big\| O_{I,J}^2(G_{2^{n+1}}^{<} f : n \in \mathbb I_{<}) \big\|_{\ell^p}
& \le 
C_{p} \|f\|_{\ell^p}.
\end{split}
\end{align}
More importantly,  the implied constant $C_p>0$ is independent of the dimension
$d\in\ZZ_+$.
\end{proposition}

\begin{proposition}
\label{pro:short}
For each $p\in(1,\infty)$ there exists $C_{p}>0$ such that for every $f\in \ell^p(\ZZ^d)$ we have 
\begin{align}
\label{id:123}
\begin{split}
\Big\|\Big( \sum_{n \in \mathbb I_{>}} V^2 \big(G_{t}^{>} f : t\in[2^n, 2^{n+1}] \big)^2\Big)^{1/2} \Big\|_{\ell^p}
&\le 
C_{p} \|f\|_{\ell^p},\\
\Big\|\Big( \sum_{n \in \mathbb I_{<}} V^2 \big(G_{t}^{<} f : t\in[2^n, 2^{n+1}] \big)^2\Big)^{1/2} \Big\|_{\ell^p}
&\le 
C_{p} \|f\|_{\ell^p}.
\end{split}
\end{align}
More importantly,  the implied constant $C_p>0$ is independent of the dimension
$d\in\ZZ_+$.
\end{proposition}

The proofs of both propositions above will rely on the Fourier
transform estimates from Sections~\ref{sec:Fouest}, \ref{sec:Festlarg}, \ref{sec:Festsmall}, results from
Section~\ref{sec:implp}, and the abstract approach to $L^p$ estimates developed by the
first author with Stein and Zorin-Kranich in \cite{MZKS1}. Similar
concepts have been discussed earlier  by the
first and third author in collaboration with Bourgain and Stein in \cite{BMSW1}. We
will also need dimension-free bounds for discrete symmetric diffusion
semigroups. Namely, for every $t>0$ we let $Q_t$  be defined on the Fourier transform side by 
 $\mathcal F_{\ZZ^d}{Q_t f} = q_t\mathcal F_{\ZZ^d}{f}$ for $f\in\ell^2(\ZZ^d)$, where
\begin{align*}
q_t(\xi):=e^{-t\sum_{i\in[d]}\sin^2(\pi\xi_i)}, \qquad  \xi\in\TT^d, \quad t>0.
\end{align*}
For further use we note that
\[
\sum_{i\in[d]}\sin^2(\pi\xi_i) \simeq |\xi|^2, \qquad \xi\in \TT^d, \quad d \ge 1.
\]
Then $\{Q_t\}_{t>0}$ is a symmetric diffusion (heat) semigroup with its generator being the discrete Laplacian on the lattice $\ZZ^d.$ This is easy to verify directly, e.g.\ along the lines of \cite[Section 4.1]{BMSW3}, where a similar statement was justified for the corresponding Poisson semigroup. In particular all the key properties, including \cite[eq. (4.7), (4.8)]{BMSW3}, continue to hold for operators corresponding to $Q_t.$  

It is well known (see for instance \cite{MSZ1} and \cite{JR} and the
references given there) that for every $p\in(1, \infty)$ there is
$C_{p}>0$ independent of $d\in\ZZ_+$ such that
\begin{align}
\label{eq:489}
\sup_{\lambda>0} \big\| \lambda N_{\lambda}(Q_t f : t>0)^{1/2} \big\|_{\ell^p}
\le 
C_{p} \|f\|_{\ell^p},
\qquad f\in\ell^p(\ZZ^d).
\end{align}
In particular, this together with \eqref{eq:688}, \eqref{estt1} and interpolation ensures 
that for every
$p\in(1, \infty)$ and $r \in (2,\infty)$ there are $C_{p,r}, C_{p}>0$
independent of $d\in\ZZ_+$ such that
\begin{align}
\nonumber
\big\|V^r(Q_t f : t>0)\big\|_{\ell^p} & \le C_{p,r} \|f\|_{\ell^p}, \qquad f\in\ell^p(\ZZ^d), \\
\label{eq:47}
\big\|\sup_{t>0}|Q_t f|\big\|_{\ell^p} & \le C_p \|f\|_{\ell^p}, \qquad f\in\ell^p(\ZZ^d).
\end{align}
On the other hand, the oscillation estimate for $Q_{t}$ follows from \cite[Theorem~3.3]{JR}. More precisely, 
for every
$p\in(1, \infty)$ there is $C_{p}>0$
independent of $d\in\ZZ_+$ such that
\begin{align} \label{id:oscQ}
\sup_{J\in\ZZ_+}\sup_{I\in \mathfrak{S}_J((0,\infty))} 	
\big\| O_{I, J}^{2}( Q_t f : t > 0) \big\|_{\ell^p} \le C_{p} \|f\|_{\ell^p}, \qquad f\in \ell^p(\ZZ^d).
\end{align}
Denoting the Littlewood--Paley ``projections'' by 
\begin{align}
\label{def:Sj}
S_j :=Q_{2^{j-1}} - Q_{2^j}, \qquad j\in \ZZ,
\end{align}
we obtain from \cite[formula (4.7)]{BMSW3} and \cite[formula
(4.8)]{BMSW3} the corresponding Littlewood--Paley decomposition
\begin{align*}
f=\sum_{j\in \ZZ} S_j f \quad\text{converges in $\ell^2(\ZZ^d)$ for every $f \in \ell^2(\ZZ^d)$,}\quad  
\end{align*}
and 
\begin{align*}
\bigg\|\Big(\sum_{j\in \ZZ}|S_j f|^2\Big)^{1/2}\bigg\|_{\ell^p}\leq C_p\|f\|_{\ell^p},
\end{align*}
where the constant $C_p>0$ depends only on $p\in (1,\infty)$, but is
independent of the dimension $d\in\ZZ_+$.

First, we will prove Proposition \ref{pro:dyad}.
\begin{proof}[Proof of Proposition~\ref{pro:dyad}]
We first establish the first inequality in 
\eqref{id:122}.
Observe that
\begin{gather*}
\sup_{\lambda>0} \big\| \lambda N_{\lambda}(G_{2^n}^{>} f : n \in \mathbb I_{>})^{1/2} \big\|_{\ell^p} \lesssim \sup_{\lambda>0} \big\| \lambda N_{\lambda}(Q_{2^n} f : n \in \mathbb I_{>})^{1/2} \big\|_{\ell^p}
+  \Big\| \Big( \sum_{n \in \mathbb I_{>}} \big| G_{2^n}^{>} f - Q_{2^n} f \big|^2\Big)^{1/2} \Big\|_{\ell^p}.
\end{gather*}
In view of \eqref{eq:489} it suffices to bound the square function in
the above inequality.  We  prove that
\begin{align}
\label{eq:9}
\Big\| \Big( \sum_{n \in \mathbb I_{>}} \big| G_{2^n}^{>} f - Q_{2^n} f \big|^2\Big)^{1/2} \Big\|_{\ell^p}\lesssim \|f\|_{\ell^p}
\end{align}
with an implicit constant independent of the dimension $d\in\ZZ_+$. We
will appeal to \cite[Theorem~2.14]{MZKS1} with $X=\ZZ^d$ endowed with
counting measure $\mathfrak m$, and the parameters $q_0=1$ and any
$1<q_1 \le 2$ (it suffices to consider only $q_1$'s close to $1$). We
take $S_j$ given by \eqref{def:Sj} and we put
\begin{align*} 
P_k=Q_{2^k} , \qquad A_k=G_{2^k}^{>} , \qquad k\in \mathbb I_{>}. 
\end{align*}
Then the maximal function $\sup_{k\in \ZZ} |P_k f|$ has a
dimension-free bound on all $\ell^p(\ZZ^d)$ spaces by
\eqref{eq:47}. Moreover, by \eqref{eq:Gtcontr} we have that
$\sup_{k\in \ZZ}\|A_k\|_{\ell^{1}\to \ell^{1}}\le 1$.  It suffices to
show that
\begin{align}
\label{id:128}
\Big\|\Big(\sum_{k\in \mathbb I_{>}}|(A_k-P_{k})S_{k+j}f|^2\Big)^{1/2}\Big\|_{\ell^2}
\lesssim
2^{-|j|/2}\|f\|_{\ell^2}.
\end{align}
Once inequality \eqref{id:128} is established, then
\cite[Theorem~2.14]{MZKS1} can be applied, since $a_j:=2^{-|j|/2}$
satisfies $\sum_{j\in\ZZ}a_j^{\frac{q_1-q_0}{2-q_0}}<\infty$, and then
inequality \eqref{eq:9} follows. To prove \eqref{id:128} we use
\eqref{id:105} and Proposition~\ref{pro:est_inf} together with
Plancherel's theorem to derive
\begin{align*}
&\Big\|\Big(\sum_{k\in\mathbb I_{>}}|(G_{2^k}^{>} - Q_{2^k})S_{k+j}f|^2\Big)^{1/2}\Big\|_{\ell^2}^2 \\
& \qquad \qquad =
\int_{\TT^d}\sum_{k\in\mathbb I_{>}} \big|(\mathcal F_{\ZZ^d}{g_{2^k}}(\xi)-q_{2^k}(\xi))
(q_{2^{k+j-1}}(\xi) - q_{2^{k+j}}(\xi))\big|^2 \cdot |\mathcal F_{\ZZ^d}{f}(\xi)|^2 \,d\xi\\
& \qquad \qquad \lesssim 
\int_{\TT^d}\sum_{k\in\mathbb I_{>}} \big|\min\big(2^k |\xi|^2,(2^k|\xi|^2)^{-1}\big) \big|^2 \cdot \big|\min\big(2^{k+j} |\xi|^{2},(2^{k+j}|\xi|^{2})^{-1}\big) \big|^2\cdot |\mathcal F_{\ZZ^d}{f}(\xi)|^2 \,d\xi \\
& \qquad \qquad \lesssim 
2^{-|j|}\|f\|_{\ell^2}^2.
\end{align*}
This implies \eqref{id:128} as desired. 

To prove the second bound in \eqref{id:122} it suffices to repeat the argument with $G_{2^{k+1}}^{<}$ in place of $G_{2^k}^{>}$ and $\mathbb I_{<}$ in place of $\mathbb I_{>}$ and instead of 
\eqref{id:105} and Proposition~\ref{pro:est_inf} we use respectively \eqref{id:106} and Proposition~\ref{pro:1}. Then the second inequality in \eqref{id:122} follows.

To prove the third inequality in \eqref{id:122}, using \eqref{eq:62}, we see that
\begin{align*}
		\sup_{J\in\ZZ_+}\sup_{I\in\mathfrak{S}_J(\mathbb{I}_{>})}
		\big\| O_{I,J}^2(G_{2^n}^{>} f : n \in \mathbb I_{>}) \big\|_{\ell^p} 
		& \lesssim
		\sup_{J\in\ZZ_+}\sup_{I\in\mathfrak{S}_J(\mathbb{I}_{>})}
\big\| O_{I,J}^2( Q_{2^n} f : n \in \mathbb I_{>}) \big\|_{\ell^p} \\
& \quad +
\Big\| \Big( \sum_{n \in \mathbb I_{>}} \big| G_{2^n}^{>} f - Q_{2^n} f \big|^2\Big)^{1/2} \Big\|_{\ell^p}.
\end{align*}
Taking into account \eqref{id:oscQ} and \eqref{eq:9} we see that the third inequality in \eqref{id:122} holds true. 
The proof of the fourth one can be derived in a similar way using \eqref{id:oscQ} and an analogue of \eqref{eq:9} with $G_{2^{n+1}}^{<}$ in place of $G_{2^{n}}^{>}$ and $\mathbb I_{<}$ in place of $\mathbb I_{>}$.
 The proof of Proposition~\ref{pro:dyad} is finished.
\end{proof} 

Before we prove Proposition \ref{pro:short} we need to establish the following important corollary.

\begin{corollary}
\label{cor:implp:Gt}
For each $p\in (1,2)$ and $\alpha\in(0,1)$ there exists a constant $C(p,\alpha)>0$, depending only on $\alpha$ and $p$, and independent of the dimension $d\in\ZZ_+$ such that  
\begin{align}
\label{eq:cor:implp:Gt}
\|G_{t+w}^{>}-G_t^{>}\|_{\ell^p\to \ell^p}
&\leq
C(p,\alpha) \left(w{t}^{-1}\right)^{\alpha},
\qquad t\ge 16,\quad w>0, \\
\label{id:120}
\big\|G_{t+w}^{<} - G_t^{<} \big\|_{\ell^p\to \ell^p}
&\leq
C(p,\alpha) \left(w{t}^{-1}\right)^{\alpha},\qquad t, w >0, \quad t+w \le 1.	 
\end{align}	
\end{corollary}

\begin{proof}
We first deal with \eqref{eq:cor:implp:Gt}. Let
$\chi \colon (0,\infty) \to [0,1]$ be a smooth function such that
\[
\ind{[16,\infty)}\le \chi\le \ind{(0,15]^c}.
\]
We apply Corollary
\ref{cor:implp} to the family of operators $M_t:=G_t  \chi(t)$, so that
$m_t:=\mathcal F_{\ZZ^d}{g_t} \chi(t)$.  Since $G_t$ is an $\ell^1(\ZZ^d)$ contraction,
see \eqref{eq:Gtcontr}, Proposition \ref{pro:est_der} implies that for
each $n\in \NN$ one has
\begin{align*}
B_n(\{m_s\}_{s>0}) \lesssim_n 1,
\end{align*}
with the implicit constant depending only on $n$ and not on the
dimension $d\in\ZZ_+$. This proves \eqref{eq:cor:implp:Gt}.
	
The proof of \eqref{id:120} is similar, observe that $\psi(16)\in(1/2, 1)$. Let us define
$\eta \colon (0,\infty) \to [0,1]$ to be a smooth function such that
\[
\ind{(0,1]}\le \eta\le \ind{[(\psi(16) + 1)/(2\psi(16)),\infty)^c}.
\]
We apply Corollary~\ref{cor:implp} with
$M_t:=G_{\psi^{-1}(\psi(16)t)} \eta(t)$, so that
$m_t:=\mathcal F_{\ZZ^d}{g}_{\psi^{-1}(\psi(16)t)} \eta(t)$. Then proceeding as in the previous case, with 
Corollary~\ref{cor:est_dersmall} (taken with
$c=(\psi(16) + 1)/2$) in place of Proposition \ref{pro:est_der}, we obtain \eqref{id:120}.
\end{proof}

We will now move on to proving Proposition \ref{pro:short}.

\begin{proof}[Proof of Proposition \ref{pro:short}] 
We will only establish the first inequality in \eqref{id:123}, the
proof of the second one can be derived with minor changes. We will
appeal to \cite[Theorem 2.39 (2) and (3)]{MZKS1} with $A_t:=G_t^{>}$, $S_{j}$ as in \eqref{def:Sj} 
and any $q_1 \in (1,2)$. We also take $1<q_0<q_1$ sufficiently close to $1$, and
$\alpha\in(0,1)$ such that $\alpha q_0 \le 1$ and satisfying
\cite[condition (2.46)]{MZKS1}, which asserts
\begin{align}
\label{eq:11}
q_0\le 2-\frac{2-q_0}{2-\alpha q_0} < q_1 \le 2.
\end{align}
In what follows we take $\alpha=1/q_0$ with $q_0$ close enough to $1$
so that \eqref{eq:11} is obviously satisfied.  Since $\{G_t\}_{t>0}$
is a family of positive linear operators, we only need to verify
\cite[conditions (2.40), (2.45) and (2.48) of Theorem
2.39]{MZKS1} to establish the first inequality in \eqref{id:123}. More precisely, we need to
verify that
\begin{align}
\nonumber
\Big\|\Big(\sum_{k\in\mathbb I_{>}} \sum_{m = 0}^{2^{l}-1}|(G_{2^k+{2^{k-l}(m+1)}}^{>} - &G_{2^k+{2^{k-l}m}}^{>})S_{j+k}f|^2 \Big)^{1/2}\Big\|_{\ell^2} \\
\label{2.40} & \lesssim 
2^{-l/2} \min\big(1, 2^{l}2^{-|j|/2}\big) \|f\|_{\ell^2},
\qquad f \in \ell^2(\ZZ^d), \quad l \ge 0, \\
\label{2.45}
\|G_{t+w}^{>} - G_{t}^{>} \|_{\ell^{q_0} \to \ell^{q_0}}  
& \lesssim 
\left(w{t}^{-1}\right)^{\alpha},\qquad t\ge 16,\quad w>0, \\ 	
\label{2.48}
\|\sup_{k\in\mathbb I_{>}}|G_{2^{k}}^{>}f|\|_{\ell^p}
& \lesssim 
\|f\|_{\ell^p}, \qquad f \in \ell^p(\ZZ^d), \quad 1<p<\infty,
\end{align}
where the implicit constants are independent of the dimension $d\in\ZZ_+$.

Observe that \eqref{2.45} follows from inequality \eqref{eq:cor:implp:Gt} of Corollary~\ref{cor:implp:Gt}
applied with $p=q_0$ and $\alpha=1/q_0$. Note that it is vital here
that the constant $C(q_0,1/q_0)$ from Corollary \ref{cor:implp:Gt}
depends only on $q_0$ and not on the dimension $d\in\ZZ_+$. Maximal inequality
\eqref{2.48} has already been proved in
Proposition~\ref{pro:dyad}, see \eqref{eq:688} and \eqref{estt1}. Once inequality \eqref{2.40} is
established, then \cite[Theorem 2.39 (2) and (3)]{MZKS1} can be
applied, since $a_{j,l} := \min(1, 2^{l}2^{-|j|/2})$ satisfies
$\sum_{l\in\NN}\sum_{j\in\ZZ}2^{-\varepsilon l}a_{j, l}^{\rho}<\infty$
for every $0<\varepsilon< \rho$ (this is precisely \cite[condition
(2.41)]{MZKS1}), and then the first inequality from \eqref{id:123} follows as desired. 
Thus, it suffices to prove \eqref{2.40}.

To  prove \eqref{2.40}, we use \eqref{id:105}, Proposition~\ref{pro:est_inf} and argue in a similar way as in the proof of Proposition~\ref{pro:dyad} to derive
\begin{align}
\label{eq:35}
\Big\|\Big(\sum_{k\in\mathbb I_{>}} \sum_{m = 0}^{2^{l}-1}\abs{(G_{2^k+{2^{k-l}(m+1)}}^{>} - G_{2^k+{2^{k-l}m}}^{>})S_{j+k}f}^2 \Big)^{1/2}\Big\|_{\ell^2}
\lesssim
2^{l/2}2^{-|j|/2}\|f\|_{\ell^2}.
\end{align}
Further, notice that Proposition~\ref{pro:est_der} implies, for $k\in\mathbb I_{>}$, the bound 
\begin{align*}
|\mathcal F_{\ZZ^d}{g}_{2^k+{2^{k-l}(m+1)}}(\xi)-\mathcal F_{\ZZ^d}{g}_{2^k+{2^{k-l}m}}(\xi)| 
\le
\int_{2^k+{2^{k-l}m}}^{2^k+{2^{k-l}(m+1)}} 
|t \partial_t \mathcal F_{\ZZ^d}{g}_t(\xi)| \,\frac{dt}{t}
\lesssim 
2^{-l},
\end{align*}
with an implicit constant  independent of the dimension $d\in\ZZ_+$.
Then by Plancherel's theorem we obtain
\begin{align}
\label{eq:37}
\begin{split}
&\Big\|\Big(\sum_{k\in\mathbb I_{>}} \sum_{m = 0}^{2^{l}-1}\abs{(G_{2^k+{2^{k-l}(m+1)}}^{>} - G_{2^k+{2^{k-l}m}}^{>})S_{j+k}f}^2 \Big)^{1/2}\Big\|_{\ell^2} \\
&\qquad =
\Big(\sum_{k\in\mathbb I_{>}} \sum_{m = 0}^{2^{l}-1}\big\|(G_{2^k+{2^{k-l}(m+1)}}^{>} - G_{2^k+{2^{k-l}m}}^{>})S_{j+k}f \big\|_{\ell^2}^2\Big)^{1/2}\\
& \qquad\lesssim \Big(\sum_{k\in\mathbb I_{>}} 2^{-l}\|S_{j+k}f\|_{\ell^2}^2\Big)^{1/2} \lesssim 2^{-l/2}\|f\|_{\ell^2}.
\end{split}
\end{align}
Combining \eqref{eq:35} and \eqref{eq:37} we obtain \eqref{2.40}, and consequently the first inequality from \eqref{id:123} follows.

To prove the second bound in \eqref{id:123} it suffices to repeat the
argument with $G_{t}^{<}$ in place of $G_{t}^{>}$ and $\mathbb I_{<}$
in place of $\mathbb I_{>}$ and instead of \eqref{id:105} and
Proposition~\ref{pro:est_inf} we use respectively \eqref{id:106} and
Proposition~\ref{pro:1}. Furthermore, instead of inequality \eqref{eq:cor:implp:Gt} of Corollary~\ref{cor:implp:Gt} we apply inequality \eqref{id:120} and instead of Proposition \ref{pro:est_der} we use Corollary \ref{cor:est_dersmall}.
Then the second inequality in \eqref{id:123}
follows and the proof is finished.
\end{proof}

\appendix

\section{} \label{sec:app}

To show Proposition~\ref{pro:hTd}, we need to prove a slight improvement of Proposition~\ref{pro:1} with explicit constants.

\begin{proposition} \label{pro:1mod}
	For every fixed  $D>0$ there are constants 
	\begin{align*}
		c_{D}  = \frac{16}{1 + \sqrt{D}},
		\qquad
		C_{D}  = \max\bigg\{ 8 \pi^2  \Big( 1 + \frac{D^2 e^{-\pi/(2D)}}{\pi^2} \Big),  
		8 \pi^3 D\Big( 1 + \frac{D^2 e^{-\pi/(2D)}}{\pi^2}  \Big) \bigg\}
	\end{align*}
	 such that
	\begin{align} \label{id:125m}
		\exp\big(-C_{D} e^{-\pi/t} |\xi|^2  \big)
		\le
		\mathcal F_{\ZZ^d}g_t(\xi)
		\le  
		\exp\big(-c_{D} e^{-\pi/t} |\xi|^2  \big),
	\end{align}
holds for all $d \in \mathbb{Z}_{+}$, $t\le D$ and $\xi\in\TT^d$.
\end{proposition}

\begin{proof}
	By the product structure of $\mathcal F_{\ZZ^d}g_t(\xi)$ it suffices to prove \eqref{id:125m} for $d=1$.
	Using \eqref{eq:ps1} and \eqref{eq:ps2} we see that 
	\begin{align*}
		1 - \frac{\theta_t(\xi)}{\theta_t(0)}=
		\theta_{1/t}(0)^{-1}\sum_{j \in \ZZ} e^{-\pi j^2/t}
		\big(1 - \cos (2\pi j \xi) \big).
	\end{align*}
	Using the facts that 
	\begin{align*}
		1 - \cos 2x = 2\sin^2 x \le 2x^2, \quad \text{for $x \in \RR$}, \quad \text{ and } \quad
		 \frac{4 x^2}{\pi^2 } \le \sin^2 x \le x^2, \quad \text{for $x \in [-\pi/2, \pi/2]$},
	\end{align*}
	we obtain
	\begin{align*}
	\theta_{1/t}(0)^{-1}	\frac{16}{\pi^2} \pi^2 \xi^2 e^{-\pi/t}
\le
	1 - \frac{\theta_t(\xi)}{\theta_t(0)}
	\le 
	\theta_{1/t}(0)^{-1} 
	4 \pi^2 \xi^2 \Big( e^{-\pi/t} + \sum_{j =2}^{\infty} e^{-\pi j^2/t} j^2 \Big).
\end{align*}
Further, using \eqref{id:101} for $t>0$ we may write
	\begin{align} \label{id:333}
 \sum_{j =2}^{\infty} e^{-\pi j^2/t} j^2 
 \le 
 \frac{t}{\pi}  \sum_{j =2}^{\infty} e^{-\pi j^2/(2t)} 
 \le
 \frac{t}{\pi}  \sum_{j =2}^{\infty} e^{-\pi j/t} 
\le
\frac{t}{\pi}   e^{-2\pi/t} \frac{1}{1 - e^{-\pi/t}} 
\le 
\frac{t^2}{\pi^2}   e^{-\pi/t} e^{-\pi/(2t)}.
\end{align}
Applying Lemma~\ref{lem:1}, we have $1 \le \theta_{1/t}(0) \le 1 + \sqrt{D}$ for $t \le D$. This together with \eqref{id:333} implies
	\begin{align} \label{id:22222}
	\frac{16 e^{-\pi/t} \xi^2 }{1 + \sqrt{D}}	 
	\le
	1 - \frac{\theta_t(\xi)}{\theta_t(0)}
	\le 
	4 \pi^2 \xi^2 e^{-\pi/t} \Big( 1 + \frac{D^2 e^{-\pi/(2D)}}{\pi^2}  \Big),
	\qquad t \le D, \quad \xi \in \TT.
\end{align}
	Taking into account the inequality $1 - x \le e^{-x}$ for $x \ge 0$, we obtain 
\begin{align*}
	\frac{\theta_t(\xi)}{\theta_t(0)}
	\le
	1 - \frac{16 e^{-\pi/t} \xi^2 }{1 + \sqrt{D}}
	\le
	\exp\Big(-\frac{16 e^{-\pi/t} \xi^2 }{1 + \sqrt{D}} \Big),
	\qquad t \le D, \quad \xi \in \TT.
\end{align*}

Finally, we show the lower bound in \eqref{id:125m}. 
Here we distinguish two cases.

\paragraph{\bf Case 1:} $4 \pi^2 e^{-\pi/t} \Big( 1 + \frac{D^2 e^{-\pi/(2D)}}{\pi^2}  \Big) \le 1/2$. Then using the inequality $1 - x \ge e^{-2x}$ for $x \in [0,1/2]$, and \eqref{id:22222} we obtain
\begin{align*}
	\frac{\theta_t(\xi)}{\theta_t(0)}
	&\ge
	1 - 4 \pi^2 \xi^2 e^{-\pi/t} \Big( 1 + \frac{D^2 e^{-\pi/(2D)}}{\pi^2}  \Big) \\
	&\ge
	\exp\bigg(-8 \pi^2 \xi^2 e^{-\pi/t} \Big( 1 + \frac{D^2 e^{-\pi/(2D)}}{\pi^2}  \Big) \bigg).
\end{align*}
This gives the lower bound in \eqref{id:125m} in Case 1.

\paragraph{\bf Case 2:} $4 \pi^2 e^{-\pi/t} \Big( 1 + \frac{D^2 e^{-\pi/(2D)}}{\pi^2}  \Big) \ge 1/2$. This case is nonempty if there is $T_0 \le D$ such that $4 \pi^2 e^{-\pi/T_{0}} \Big( 1 + \frac{D^2 e^{-\pi/(2D)}}{\pi^2}  \Big) = 1/2$. Then our aim is to show the lower bound in \eqref{id:125m} for $t \in [T_0,D]$. Here we proceed in a different way than in Case 1. We use Lemma~\ref{lem:4} and we obtain
\begin{align*}
	\frac{\theta_t(\xi)}{\theta_t(0)}
	\ge
	\exp\big(-\pi t \xi^2\big) 
	\ge
	\exp\bigg(-8 \pi^3 D \xi^2 e^{-\pi/t} \Big( 1 + \frac{D^2 e^{-\pi/(2D)}}{\pi^2}  \Big) \bigg), 
		\qquad t \in [T_0,D], \quad \xi \in \TT.
\end{align*} 
The latter inequality is true because we have 
$\sup_{t \in [T_0,D]} \pi t e^{\pi/t} \le \pi D e^{\pi/T_{0}} 
=  8 \pi^3 D\Big( 1 + \frac{D^2 e^{-\pi/(2D)}}{\pi^2}  \Big)$.
This gives the lower bound in \eqref{id:125m} in Case 2.
	This finishes the proof of Proposition~\ref{pro:1mod}.
\end{proof}

Now we are ready to prove Proposition~\ref{pro:hTd}.

\begin{proof}[Proof of Proposition~\ref{pro:hTd}.]
Taking into account \eqref{HKT} and \eqref{eq:ft2}, we see that 
\begin{align*}
	\frac{ H_t(\xi) }{ H_t(0) } 
	=
	\frac{ \Theta_{1/(4\pi t)} (\xi) }{ \Theta_{1/(4\pi t)} (0) } 
	=
	\mathcal F_{\ZZ^d}{g}_{1/(4\pi t)} (\xi).
\end{align*}
Using Propositions~\ref{pro:est_inf} we get the upper bound in \eqref{HKT1}. 
Observe that the lower bound in \eqref{HKT1} follows directly from the lower bound in Lemma~\ref{lem:4} and the product structure of $\Theta_{t}$, see \eqref{eq:Theta_def}.
Therefore our task is to show \eqref{HKT2}.  This, however, is a straightforward consequence of Proposition~\ref{pro:1mod} applied with $D=32$.
The proof of Proposition~\ref{pro:hTd} is finished.
 
\end{proof}


\begin{thebibliography}{100}
	
\bibitem{Ald1} 
J.M.\ Aldaz, 
\textit{The weak type $(1, 1)$ bounds for
the maximal function associated to cubes grow to infinity with the
dimension.}  Ann.\  Math.\  173 (2011), 1013--1023.

\bibitem{AB_book} 
C.D.\ Aliprantis, O.\ Burkinshaw,
\textit{Principles of Real Analysis.}  
Academic Press, San Diego, 1998.


\bibitem{Bod} 
U.T.\ B\"{o}dewadt,
\textit{Die Kettenregel f\"ur h\"ohere Ableitungen.}  
Math.\ Z.\ 48 (1943), 735--746.

\bibitem{B4} 
J.\ Bourgain, 
\textit{On high dimensional maximal
  functions associated to convex bodies.} 
  Amer.\ J.\  Math.\ 108
(1986), 1467--1476.
        
\bibitem{B5} 
J.\ Bourgain, 
\textit{On $L^p$ bounds for maximal
  functions associated to convex bodies in $\RR^n$.} 
  Israel J.\ Math.\ 54 (1986), 257--265.

\bibitem{B6} 
J.\ Bourgain, 
\textit{On the Hardy--Littlewood maximal
  function for the cube.} 
  Israel J.\ Math.\  203 (2014), 275--293.


\bibitem{BMSW1} 
J.\ Bourgain, M.\ Mirek, E.M.\ Stein, B.\ Wr\'obel,
\textit{Dimension-free variational estimates on $L^p(\RR^d)$ for
symmetric convex bodies.} Geom.\ Funct.\ Anal.\ (1) 28 (2018), 58--99.
	
\bibitem{BMSW2} 
J.\ Bourgain, M.\ Mirek, E.M.\ Stein, B.\ Wr\'obel,
\textit{On discrete Hardy--Littlewood maximal functions 
	over the balls in
	$\ZZ^d$: dimension-free estimates.} 
Geometric Aspects of Functional Analysis -- Israel Seminar (GAFA) 2017-2019, Lecture Notes in Mathematics 2256.
	
	\bibitem{BMSW3} 
	J.\ Bourgain, M.\ Mirek, E.M.\ Stein, B.\ Wr\'obel,
	\textit{Dimension-free estimates for discrete Hardy--Littlewood
		averaging operators over the cubes in
		$\mathbb Z^d$.} 
		Amer.\ J.\ Math.\ 141 (2019), 857--905.
	
	\bibitem{BMSW4} 
	J.\ Bourgain, M.\ Mirek, E.M.\ Stein, B.\ Wr\'obel,
	\textit{On the Hardy--Littlewood maximal functions in high dimensions: Continuous and discrete perspective.} Geometric Aspects of Harmonic Analysis.
	A conference proceedings on the occasion of Fulvio Ricci's 70th birthday, Cortona, Italy, 25-29.06.2018. Springer INdAM Series 45, (2021), 107--148.	
	





\bibitem{Car1} 
A.\ Carbery,
\textit{An almost-orthogonality principle
  with applications to maximal functions associated to convex bodies.}
Bull.\ Amer.\ Math.\ Soc.\  14 (1986), no. 2,  269--274.


\bibitem{DGM1} 
L.\ Deleaval, O.\ Gu\'edon, B.\ Maurey,
\textit{Dimension-free bounds for the Hardy-Littlewood maximal operator associated to convex sets.} 
Ann.\ Fac.\ Sci.\ Toulouse Math.\ 27 (2018), 1--198.
	
\bibitem{Gra1} 
L.\ Grafakos, 
Classical Fourier Analysis (Third edition), Graduate Texts in Mathematics 249, Springer 2014.
	
\bibitem{JR} 
R.L.\ Jones, K.\ Reinhold,
\textit{Oscillation and variation inequalities for convolution powers.}
Ergodic Theory Dynam.\ Systems 21 (2001), 1809--1829.	
	
\bibitem{JSW} 
R.L.\ Jones, A.\ Seeger, J.\ Wright,
\textit{Strong variational and jump inequalities in harmonic analysis.}
Trans.\ Amer.\ Math.\ Soc.\ 360 (2008),  6711--6742.


\bibitem{KMPW} 
D.\ Kosz,  M.\ Mirek, P.\ Plewa, B.\ Wr{\'o}bel,
\textit{Some remarks on dimension-free estimates for the discrete Hardy--Littlewood maximal functions.}
Israel J.\ Math.\ 254 (2023), 1--38.

	
\bibitem{MZKS1} 
M.\ Mirek, E.M.\ Stein, P.\ Zorin-Kranich,  
\textit{A bootstrapping approach to jump inequalities and their applications.} 
Anal.\ PDE (2) 13 (2020), 527--558.

\bibitem{MSZ1} 
M.\ Mirek, E.M.\ Stein, P.\ Zorin-Kranich,  
\textit{Jump inequalities via real interpolation.}
Math.\ Ann.\ 376 (2020), 797--819.

\bibitem{MiSlSz} 
M.\ Mirek, W.\ S{\l}omian, T.Z.\ Szarek, 
\textit{Some remarks on oscillation inequalities}, 
Ergodic Theory Dynam.\ Systems 43 (2023), 3383--3412.

\bibitem{MSW} 
M.\ Mirek, T.Z.\ Szarek, J.\ Wright,
\textit{Oscillation inequalities in ergodic theory and analysis: one-parameter and multi-parameter perspectives.} 
Rev.\ Mat.\ Iberoam.\ 38 (2022), 2249--2284.

\bibitem{MiSzWr} 
M.\ Mirek, T.Z.\ Szarek, B.\ Wr\'obel,
\textit{Dimension-free estimates for the discrete spherical maximal functions.}
Int.\ Math.\ Res.\ Not.\ IMRN 2024 (2024), 901--963. 


\bibitem{Mul1} 
D.\ M\"uller,
\textit{A geometric bound for maximal functions associated to convex bodies.}
Pacific J.\ Math.\  142 (1990), no. 2,  297--312.

\bibitem{NW} 
J.\ Niksiński, B.\ Wr\'obel,
\textit{Dimension-free estimates for discrete maximal functions and lattice points in high-dimensional spheres and balls with small radii.} 
J.\ Math.\ Pures Appl.\, 214 (2026), 1--60.



	\bibitem{NIST} 
	F.W.J.\ Olver,  D.W.\ Lozier, R.F.\ Boisvert, C.W.\ Clark (editors), 
	\textit{NIST Handbook of Mathematical Functions}, U.S. Department of Commerce, National Institute of Standards and Technology, Washington, DC; Cambridge University Press, Cambridge, 2010. xvi+951 pp.


        \bibitem{bigs} E.M.\ Stein,
Harmonic Analysis: Real-Variable Methods, Orthogonality, and Oscillatory Integrals.
Princeton University Press (1993).


          \bibitem{SteinRiesz} 
          E.M.\ Stein,
  \textit{Some results in harmonic analysis in $\mathbb R^n$, $n\to \infty$.}
Bull.\ Amer.\ Math.\ Soc.\  9 (1983),  71--73.


\bibitem{SteinMax} 
 E.M.\ Stein,
\textit{The development of square functions in the work of A.\ Zygmund.}
Bull.\  Amer.\ Math.\ Soc.\  7 (1982),  359--376.

        
	\bibitem{Ste1} 
	E.M.\ Stein, 
	\textit{Topics in harmonic analysis
		related to the Littlewood--Paley theory.}
	Annals of Mathematics
	Studies, Princeton University Press 1970, 1--157.


\bibitem{SteinPR} 
E.M.\ Stein, 
\textit{Private communication}, Bonn 2016.

\bibitem{StStr}  E.M.\ Stein, J.O.\ Str\"omberg,
\textit{Behavior of maximal functions in $\mathbb{R}^n$ for large $n$.} 
Ark.\ Mat.\  21 (1983), 259--269.





        
\end{thebibliography}
\end{document}